\documentclass[a4paper,10pt]{amsart}
\usepackage[a4paper, textwidth=14cm]{geometry}

\usepackage{marvosym}
\usepackage{amsmath,amsthm,amssymb,amsfonts,extarrows}
\usepackage{enumerate,amscd,amsxtra,MnSymbol}
\usepackage{mathrsfs}
\usepackage{color}
\usepackage[cmtip,all]{xy}
\usepackage{eucal}
\usepackage{bbold}
\usepackage{ upgreek }
\usepackage{paralist}
\usepackage{tikz-cd}

\theoremstyle{plein}
\newtheorem{theorem}{Theorem}[section]
\newtheorem*{theorem*}{Theorem}
\newtheorem{lemma}[theorem]{Lemma}
\newtheorem{proposition}[theorem]{Proposition}
\newtheorem*{proposition*}{Proposition}
\newtheorem{corollary}[theorem]{Corollary}
\newtheorem*{corollary*}{Corollary}

\newtheorem{conjecture*}{Conjecture}

\theoremstyle{definition}
\newtheorem{construction}[theorem]{Construction}

\newtheorem{notation}[theorem]{Notation}

\newtheorem{example}[theorem]{Example}
\newtheorem{definition}[theorem]{Definition}
\newtheorem{definition*}{Definition}
\newtheorem{remark}[theorem]{Remark}
\newtheorem{remark*}{Remark}

\newcommand{\bQ}{{\mathbb Q}}

\newcommand{\mA}{{\mathcal A}}
\newcommand{\mB}{{\mathcal B}}
\newcommand{\mC}{{\mathcal C}}
\newcommand{\mD}{{\mathcal D}}
\newcommand{\mE}{{\mathcal E}}

\newcommand{\mL}{{\mathcal L}}
\newcommand{\mM}{{\mathcal M}}
\newcommand{\mN}{{\mathcal N}}

\newcommand{\mP}{{\mathcal P}}
\newcommand{\mQ}{{\mathcal Q}}
\newcommand{\mR}{{\mathcal R}}
\newcommand{\mS}{{\mathcal S}}

\newcommand{\mU}{{\mathcal U}}
\newcommand{\mV}{{\mathcal V}}
\newcommand{\mW}{{\mathcal W}}
\newcommand{\mX}{{\mathcal X}}
\newcommand{\mY}{{\mathcal Y}}

\newcommand{\A}{{\mathrm A}}
\newcommand{\B}{{\mathrm B}}
\newcommand{\C}{{\mathrm C}}

\newcommand{\F}{{\mathrm F}}
\newcommand{\G}{{\mathrm G}}
\newcommand{\rH}{{\mathrm H}}

\newcommand{\K}{{\mathrm K}}
\renewcommand{\L}{{\mathrm L}}

\newcommand{\N}{{\mathrm N}}

\newcommand{\W}{{\mathrm W}}
\newcommand{\X}{{\mathrm X}}
\newcommand{\Y}{{\mathrm Y}}
\newcommand{\Z}{{\mathrm Z}}

\newcommand{\rc}{{\mathrm c}}

\newcommand{\bj}{{\mathrm j}}
\newcommand{\bi}{{\mathrm i}}
\newcommand{\m}{{\mathrm m}}
\newcommand{\bk}{{\mathrm k}}

\newcommand{\q}{{\mathrm q}}
\newcommand{\g}{{\mathrm g}}
\newcommand{\n}{{\mathrm n}}

\newcommand{\op}{\mathrm{op}}
\newcommand{\dual}{\vee}

\newcommand{\Ho}{\mathsf{Ho}}

\newcommand{\f}{\mathsf{f}}

\newcommand{\colim}{\mathrm{colim}}

\newcommand{\ot}{\otimes}

\newcommand{\h}{\mathrm{h}}

\newcommand{\id}{\mathrm{id}}

\newcommand{\Cat}{\mathsf{Cat}}

\newcommand{\Set}{\mathsf{Set}}

\renewcommand{\Pr}{\mathrm{Pr}}

\newcommand{\Fun}{\mathrm{Fun}}

\newcommand{\cocart}{{\mathrm{cocart}}}
\newcommand{\Cocart}{{\mathrm{Cocart}}}
\newcommand{\cart}{{\mathrm{cart}}}
\newcommand{\Cart}{{\mathrm{Cart}}}

\newcommand{\Tw}{{\mathrm{Tw}}}

\newcommand{\rev}{{\mathrm{rev}}}

\newcommand{\Env}{{\mathrm{Env}}}

\newcommand{\bZ}{{\mathbb{Z}}}

\newcommand{\Act}{{\mathrm{Act}}}

\newcommand{\Mor}{{\mathrm{Mor}}}

\newcommand{\Corr}{{\mathrm{Corr}}}

\newcommand{\CART}{{\mathrm{CART}}} 
\newcommand{\CAT}{{\mathrm{CAT}}} 
\newcommand{\Lax}{{\mathrm{Lax}}}

\newcommand{\CORR}{{\mathrm{CORR}}} 
\newcommand{\PR}{{\mathrm{PR}}} 
\newcommand{\FUN}{{\mathrm{FUN}}} 
 
\usepackage{pdfpages}

\title{A local-global principle for parametrized $\infty$-categories} % of $\infty$-categories}

\begin{document}
	
\author{Hadrian Heine, \\ University of Oslo, Norway, \\ hadriah@math.uio.no}
\maketitle

\begin{abstract}
We prove a local-global principle for parametrized $\infty$-categories: 
we show that any functor $\mathcal{B} \to \mathcal{C}$ is determined by the following data: the collection of fibers $\mathcal{B}_X$ for $X$ running through the set of equivalence classes of objects of $\mathcal{C}$ endowed with the action of the space of automorphisms $\mathrm{Aut}_X(\mathcal{B})$ on the fiber, the local data, together with a locally cartesian fibration $\mD \to \mathcal{C}$ and $\mathrm{Aut}_X(\mathcal{B})$-linear equivalences $\mD_X \simeq \mP(\mathcal{B}_X)$ to the $\infty$-category of presheaves on $\mathcal{B}_X$, the gluing data. 
As applications we compute the mapping spaces of the conditionally existing internal hom of $\infty\Cat_{/\mathcal{C}}$
and extend the $\infty$-categorical Grothendieck-construction by proving that
$\infty$-categories over any $\infty$-category $\mathcal{C}$
are classified by normal lax 2-functors to a double $\infty$-category of correspondences.

\end{abstract}

\tableofcontents

\section{Introduction}

%Hasse's local-global principle is one of the guiding principles in arithmetic and 

By Hasse's famous local-global principle %or principle about the arithmetic fracture square
%asserts 
the natural commutative square
$$
\begin{xy}
\xymatrix{
\bZ \ar[d]
\ar[rr]
&&\bQ \ar[d]
\\
\prod_{p} \bZ_p  \ar[rr] && (\prod_{p} \bZ_p) \ot_\bZ \bQ}
\end{xy}$$
relating the integers, rationals and $p$-adic numbers, is a cartesian square, known as the arithmetic fracture square.
%The latter square, also known as the arithmetic fracture square, 
Hasse's principle serves as a guiding principle in arithmetic and arithmetic geometry 
and was extended to a fundamental principle in unstable and stable homotopy theory fracturing an appropriate finite (stable) homotopy type into its rationalization and $p$-completion.

It is goal of this article to prove a local-global principle for parametrized $\infty$-categories.
It is a fundamental insight of \cite{BarwickParametrized}, \cite{barwick2016parametrizedhighercategorytheory},  \cite{nardinparametrizedequivarianthigheralgebra}, \cite{ShahParametrized}, \cite{shahparametrizedhighercategorytheory}
that parametrized $\infty$-category theory, i.e. the systematic study of $\infty$-categories over a fixed base $\infty$-category, is a powerful tool to perform constructions in equivariant and motivic homotopy theory \cite{barwickfibrewiseeffectiveburnsideinftycategory}, \cite{quigleyparametrizedtate}, \cite{cnossen2024normedequivariantringspectra}, \cite{Bachmann2017NormsIM}. 
%This comes primarily from the fact that genuine actions of a group on some space or $\infty$-category are best understood as cocartesian fibrations to the orbit category of the group, a presentation that which give rise to various examples of parametrized $\infty$-categories
%that are not necessarily cocartesian fibrations anymore.
Parametrized $\infty$-category theory relies on a deep understanding and efficient control of fibrations of $\infty$-categories.
The latter play the role in higher category theory that Kan-fibrations play in homotopy theory
but the complexity is drastically higher \cite{MR4074276}.
%relating homotopical universal constructions to tractable combinatorial and geometrical models. 
%Similarly, work of Lurie \cite{lurie.HTT}, \cite{lurie.higheralgebra} and Ayala-Francis \cite{MR4074276} demonstrate that higher category theory heavily relies on a deep understanding of fibrations of $\infty$-categories and manipulating fibrations is the key technique when working sucessfully with $\infty$-categories.
%Extending the theory of Kan-fibrations 
The next diagram gives an impression of the plethora of fibrations of $\infty$-categories, where all sub-squares are cartesian:
$$
\begin{xy}
\xymatrix{
&  \{ \mathrm{Kan} \ \mathrm{fibrations} \}
\ar[rd] \ar[ld] \ar[d]
\\
\{ \mathrm{Left} \ \mathrm{fibrations} \} \ar[rd] \ar[d] &   \{ {\mathrm{Bicartesian} \ \mathrm{fibrations}} \} \ar[ld]\ar[rd] & \{ \mathrm{Right} \ \mathrm{fibrations} \} \ar[ld] \ar[d] 
\\
\{ \mathrm{Cocartesian} \ \mathrm{fibrations} \}  \ar[rd] \ar[d] & \{ \underset{\mathrm{whose} \ \mathrm{fibers} \ \mathrm{are} \ \mathrm{spaces}}{\mathrm{Exponential} \ \mathrm{fibrations}} \}\ar[d] & \{ \mathrm{Cartesian} \ \mathrm{fibrations} \}  \} \ar[ld] \ar[d]\\
\{ \underset{\mathrm{fibrations} }{\mathrm{Locally} \ \mathrm{cocartesian} \} } \ar[rd] & \{ \mathrm{Exponential} \ \mathrm{fibrations} \} \ar[d] & \underset{\mathrm{fibrations} }{\{\mathrm{Locally} \ \mathrm{cartesian} \} }  \ar[ld] \\
& \{ \mathrm{parametrized}\ \infty\mathrm{-categories} \} &
}
\end{xy} $$

All fibrations in the latter diagram have a common base $\infty$-category. Choosing the base to be a space the latter diagram degenerates and only two types of fibrations remain:
Kan-fibrations and bicartesian fibrations.
Since any space splits as the disjoint union of path components, the first type of fibration
classifies a family of spaces, the family of fibers over the path components, equipped with the action of the loop space.
Similarly, the second type of fibration classifies a family of $\infty$-categories equipped with the action of the loop space.

We prove a local-global principle for parametrized $\infty$-categories following the next table of analogies:
$$\begin{tabular}{ l | c}
\hline			
$\bZ$ & \{$\infty$-categories over the base\} \\
\hline
$\bQ $ & \{Locally cartesian fibrations whose fibers are presentable \\ & and whose fiber transports admit a right adjoint\} \\
\hline
primes $p$ & equivalence classes $p$ of objects of the base \\
\hline
$\bZ_p $ & $ \infty\Cat[\Omega_p]$, the $\infty$-category of small $\infty$-categories \\ & equipped with an action of the loop space at $p$. \\
\hline
$\bQ_p = \bQ \ot_\bZ \bZ_p $ & $\Pr^L[\Omega_p]$, the $\infty$-category of presentable $\infty$-categories \\ & equipped with an action of the loop space at $p$. 
\end{tabular}$$

We prove the following theorem, where for any small $\infty$-category $\mC$ we write $\pi_0(\mC)$ for the set of equivalence classes in $\mC$
and $ \mathrm{LoCART}^L_\mC $ for the $\infty$-category of locally cartesian fibrations whose fibers are presentable and whose fiber transports admit a right adjoint and arbitrary functors over
$\mC$ inducing left adjoints on every fiber:
\begin{theorem}\label{0}(Theorem \ref{corr})
For every small $\infty$-category $\mC$ there is a pullback square of $(\infty,2)$-categories
$$
\begin{xy}
\xymatrix{
\infty\Cat_{/\mC} \ar[d] \ar[rr]
&& \mathrm{LoCART}^L_\mC \ar[d]^{}
\\
\prod_{p \in \pi_0(\mC)} \infty\Cat[\Omega_p]\ar[rr] && \prod_{p \in \pi_0(\mC)} \Pr^L[\Omega_p].}
\end{xy}$$
\end{theorem}

Theorem \ref{0} is useful since it reduces the theory of parametrized $\infty$-categories to the theory of locally cartesian fibrations whose fibers are presentable and fiber transports admit a right adjoint.
This correspondence restricts to several classes of fibrations of interest:
Theorem \ref{0} reduces the theory of exponential fibrations to the theory of cartesian fibrations whose fibers are presentable and fiber transports admit a right adjoint (Proposition \ref{expo}). 
It reduces the theory of cocartesian fibrations to the theory of bicartesian fibrations whose fibers are presentable and fiber transports admit a right adjoint (Lemma \ref{lemar}).
In particular, Theorem \ref{0} relates several classes of relevant fibrations of $\infty$-categories to each other and therefore helps to understand their relationships.

Parametrized $\infty$-category theory might be thought of as a generalization
of $\infty$-category theory and parametrized homotopy theory that studies $\infty$-categories over  a fixed $\infty$-category $\mC.$ For $\mC=[1]$ the walking arrow parametrized $\infty$-category theory specializes to the theory of correspondences of $\infty$-categories, which are a crucial concept in the theory of dualities of $\infty$-categories \cite{heine2024infinitycategoriesdualityhermitian}.
We apply Theorem \ref{0} for $\mC =[1]$ to in \cite[Theorem 5.21.]{heine2024infinitycategoriesdualityhermitian} and \cite[Theorem 6.27.]{heine2021realktheorywaldhauseninfinity} to prove an equivalence between symmetric bilinear functors and dualities.

We use Theorem \ref{0} to construct a Grothendieck construction for parametrized $\infty$-categories.
The Grothendieck construction for cocartesian and cartesian fibrations of $\infty$-categories 
\cite[Theorem 3.2.0.1.]{lurie.HTT}, \cite{nuiten2023straightening}, \cite{moser2023inftyncategoricalstraighteningunstraighteningconstruction}, \cite{abellán2024straighteninglaxtransformationsadjunctions}
is a cornerstone in higher category theory and an essential and crucial tool to construct functors to the 
$\infty$-category $\infty\Cat$ of small $\infty$-categories that has ubiquitious applications.
Due to its importance and relevance several extensions of the Grothendieck-construction have been studied: \cite[\S 3]{luriegoodwillie}, \cite[Theorem B.4.3.]{Ayala2019StratifiedNG} offer a 
Grothendieck construction for locally cartesian fibrations. 
\cite[Theorem 0.8]{MR4074276}, \cite[4.8. Proposition]{barwick2016fibrationsinftycategorytheory}
construct a Grothendieck-construction for exponential fibrations.

Since Theorem \ref{0} reduces the theory of parametrized $\infty$-categories to the theory of locally cartesian fibrations, we can apply Theorem \ref{0} to extend the Grothendieck construction for locally cartesian fibrations to a Grothendieck construction for parametrized $\infty$-categories,
which extends all Grothendieck constructions.
While the Grothendieck-construction for cartesian fibrations classifies functors to the $\infty$-category $\infty\Cat$ of small $\infty$-categories and functors, the Grothendieck-construction for exponential fibrations classifies functors to the $\infty$-precategory $\Corr$ of small $\infty$-categories and correspondences.
We extend all Grothendieck constructions via the concept of double $\infty$-category \cite[Definition 2.4.3.]{MR3345192}, which is a category object
%, i.e. Segal object, 
in $\infty\Cat$, and a 2-categorical refinement of an $\infty$-precategory. % and so a categorical enhancement of a Segal space.
We use Theorem \ref{0} to construct a double $\infty$-category $\CORR$ (Theorem \ref{corre}) that refines the $\infty$-precategory $\Corr$ of correspondences, thereby answering a question of \cite[Question 0.14]{MR4074276}.
We apply Theorem \ref{0} to prove the following Grothendieck-construction for parametrized $\infty$-categories that classifies lax normal functors to the double $\infty$-category $\CORR,$
where $\N$ is the (derived version of the) nerve of an $\infty$-category:
%We prove the following corollary, where we view an $\infty$-category canonically as an $\infty$-precategory and so as a Segal space and double $\infty$-category:
%\begin{corollary}\label{99}(Corollary \ref{genstr}) Let $\mC$ be a small $\infty$-category.
%There is a canonical equivalence of $\infty$-categories $$ \FUN(\N(\mC^\op), \CORR) \simeq \mathrm{EXP}_{\mC}$$	
%between the $\infty$-category of maps of double $\infty$-categories $\mC^\op \to \CORR$ and the
%full subcategory $\mathrm{EXP}_{\mC} \subset \infty\Cat_{/\mC}$ spanned by the exponential fibrations over $\mC$.
%\end{corollary}
%Corollary \ref{99} in different formulation is also claimed without proof in \cite[4.8. Proposition]{barwick2016fibrationsinftycategorytheory}. 
\begin{theorem}\label{2} (Corollary \ref{genstr}) Let $\mC$ be a small $\infty$-category.
There is a canonical equivalence of $\infty$-categories $$ \Lax\Fun(\N(\mC^\op), \CORR) \simeq \infty\Cat_{/\mC},$$	
where the left hand side is the $\infty$-category of lax normal functors.
% of double $\infty$-categories $\mC^\op \to \CORR$.
\end{theorem}

%Theorem \ref{2} was very recently also deduced \cite{blomstraighteningfunctor} from a universal property of the double $\infty$-category of spans \cite[Corollary 3.11.]{spans}.
Although Theorem \ref{2} is interesting from a theoretical point of view, one might wonder
why a Grothendieck construction for parametrized $\infty$-categories is useful, and what kind of lax normal functors to $\CORR$ one might like to classify. We offer an instructive example (Example \ref{exoo}).
Parametrized $\infty$-category theory studies for every
small $\infty$-category $\mC$ the $\infty$-category $\infty\Cat_{/\mC}$ of small $\infty$-categories over $\mC.$
The $\infty$-category $\infty\Cat_{/\mC}$ is not cartesian closed. However by the defining property of exponentiability for every exponential fibration $\mA \to \mC$ the functor
$\mA \times_\mC (-) :  \infty\Cat_{/\mC} \to \infty\Cat_{/\mC}$ admits a right adjoint
$\Fun^\mC(\mA,-).$
%Parametrized $\infty$-category theory might be thought of as a generalization
%of $\infty$-category theory and parametrized homotopy theory that specializes to the theory of correspondences, $\infty$-categories over the walking arrow $[1].$
%Correspondences are a useful concept categorifying the concept of relation and present in higher geometry. 
%Theorem \ref{0} for $\mC=[1]$ identies correspondences $\mM \to [1]$ with cartesian fibrations $\mD \to [1]$ such that $\mD_0 \simeq \mP(\mM_0), \mD_1 \simeq \mP(\mM_1).$
%By the $\infty$-categorical Grothendieck construction the latter classifies a left adjoint functor
%$ \mP(\mM_1) \to \mP(\mM_0)$ whose restriction along the Yoneda-embedding
%gives a functor $\mM_0^\op \times \mM_1 \to \mS$ known as a pro-functor $\mM_1 $ to $\mM_0.$
%Consequently, Theorem \ref{0} for $\mC=[1]$ recovers the well-known equivalence between correspondences and pro-functors but gives a global version of this equivalence (Theorem \ref{lfib}).
%It is an important insight of ... that the $\infty$-category $\infty\Cat_{/[1]}$ of correspondences is cartesian closed. The internal hom $\Fun^{[1]}(\mC,\mD)\to [1] $ associated to two correspondences
%$\mC \to [1], \mD\to [1]$ plays the role in the theory of correspondences that the $\infty$-category of functors plays in $\infty$-category theory. For that reason it is a central and ubiquitious construction when working with correspondences. However 
For every functor $\mB \to \mC$ the internal hom $\Fun^{\mC}(\mA,\mB) \to \mC $ is a parametrized $\infty$-category that plays the role of the $\infty$-category of functors in higher category theory,
and therefore is a crucial and ubiquitious construction in parametrized $\infty$-category theory,
which is unfortunately generally hard to understand.
The internal hom is generally not a cocartesian, cartesian or even exponential fibration and therefore cannot be analyzed by means of the usual Grothendieck-construction.
%used to produce a functor $\mC \to \infty\Cat$ or $\mC \to \Corr.$
%analyzed via the usual Grothendieck-constructions.
%, and so requires a 
However, we can apply Theorem \ref{2} to classify the internal hom $\Fun^{\mC}(\mA,\mB)\to \mC $
by a lax normal functor $\mC^\op \to  \CORR$, which can be used to analyze the internal hom.
The classified lax normal functor $\mC^\op \to  \CORR$ sends any morphism $\theta: X \to Y$ in $\mC$
to a pro-functor
$$  \Fun^{\mC}(\mA,\mB)_Y \simeq \Fun(\mA_Y,\mB_Y) \to \Fun^{\mC}(\mA,\mB)_X \simeq \Fun(\mA_X,\mB_X)$$ corresponding to a functor
$\gamma: \Fun(\mA_X,\mB_X)^\op \times \Fun(\mA_Y,\mB_Y)  \to \mS$.
The functors $\mA \to \mC, \mB \to \mC$ classify lax normal functors
$\mC^\op \to \CORR$ that send the morphism $\theta: X \to Y$ to pro-functors
$ \A_Y \to \mA_X, \mB_Y \to \mB_X$ corresponding to functors $\alpha: \mA_X^\op \times \mA_Y \to \mS, \beta: \mB_X^\op \times \mB_Y \to \mS.$
The functor $\gamma$ sends any pair $(F,G) \in \Fun(\mA_X,\mB_X)^\op \times \Fun(\mA_Y,\mB_Y)$
to the limit of the functor $$\beta \circ (F^\op \times G) \circ \rho: \mW \to \mS,$$ where
$\rho: \mW \to  \mA_X^\op \times \mA_Y$ is the left fibration classified by $\alpha$.
We use the latter description of the internal hom in $\infty\Cat_{/[1]}$ in \cite[Proposition 3.22.]{heine2024equivalencerealshermitian} to prove an equivalence between a real $S$-construction and hermitian $\mQ$-construction.

\newpage

\section{Notation and terminology}

%\subsection{Notation and terminology}

We fix a hierarchy of Grothendieck universes whose objects we call small, large, very large, etc.
We call a space small, large, etc. if its set of path components and its homotopy groups are for any choice of base point. We call an $\infty$-category small, large, etc. if its maximal subspace and all its mapping spaces are.

\vspace{2mm}

We write 
\begin{itemize}
\item $\Set$ for the category of small sets.
\item $\Delta$ for (a skeleton of) the category of finite, non-empty, partially ordered sets and order preserving maps, whose objects we denote by $[\n] = \{0 < ... < \n\}$ for $\n \geq 0$.
\item $\mS$ for the $\infty$-category of small spaces.
\item $ \infty\Cat$ for the $\infty$-category of small $\infty$-categories.
\item $\infty\Cat^{\rc \rc} $ for the $\infty$-category of large $\infty$-categories with small colimits and small colimits preserving functors.
\end{itemize}

\vspace{2mm}

We often indicate $\infty$-categories of large objects by $\widehat{(-)}$, for example we write $\widehat{\mS}, \infty\widehat{\Cat}$ for the $\infty$-categories of large spaces, $\infty$-categories.

For any $\infty$-category $\mC$ containing objects $\A, \B$ we write
\begin{itemize}
\item $\mC(\A,\B)$ for the space of maps $\A \to \B$ in $\mC$,
\item $\mC_{/\A}$ for the $\infty$-category of objects over $\A$,
\item $\Ho(\mC)$ for its homotopy category,
\item $\mC^{\triangleleft}, \mC^{\triangleright}$ for the $\infty$-category arising from $\mC$ by adding an initial, final object, respectively,
\item $\iota(\mC) $ for the maximal subspace in $\mC$.
\end{itemize}

Note that $\Ho(\infty\Cat)$ is cartesian closed and for small $\infty$-categories $\mC,\mD$ we write $\Fun(\mC,\mD)$ for the internal hom, the $\infty$-category of functors $\mC \to \mD.$ 
%\subsubsection*{Inclusions and embeddings}

We often call a fully faithful functor an embedding.
We call a functor an inclusion %(of a subcategory)
if %one of the following equivalent conditions holds:
%\begin{itemize}\item For any $\infty$-category $\mB$ the induced map$\Cat_\infty(\mB,\mC) \to \Cat_\infty(\mB,\mD)$ is an embedding.\item The functor $\phi: \mC \to \mD$ 
if it induces an embedding on maximal subspaces and on all mapping spaces.
%\item The functor $\Ho(\phi):\Ho(\mC) \to \Ho(\mD)$ is an inclusion and the functor $\mC \to \Ho(\mC) \times_{\Ho(\mD)} \mD$ is an equivalence.\end{itemize}

%or say that $\phi$ exhibits $\mC$ as a subcategory of $\mD$

%In this case $\phi$ is uniquely determined by $\mD$ and $\Ho(\phi): \Ho(\mC) \to \Ho(\mD).$\vspace{2mm}

%We call a monoidal $\infty$-category compatible with small colimits if it admits small colimits, which are preserved by the tensor product in each component.
%We call a monoidal $\infty$-category presentable if it is compatible with small colimits and its underlying $\infty$-category is presentable.
%\vspace{2mm}

%\subsubsection*{Relative cocartesian fibrations}

%Let $\mA$ be an $\infty$-category and $\mE \subset \Fun([1],\mA)$ a full subcategory. We call a functor $\X \to \mA$ a cocartesian fibration relative to $\mE$ if for every morphism $[1] \to \mA$ that belongs to $\mE$ the pullback$[1] \times_\mA \X \to [1]$ is a cocartesian fibration whose cocartesian morphisms are preserved by the projection $[1] \times_\mA \X \to \X.$We call a functor $\X \to \Y$ over $\mA$ a map of cocartesian fibrations relative to $\mE$ if it preserves cocartesian lifts of morphisms of $\mE.$We write $\Cat_{\infty/\mA}^\mE \subset \Cat_{\infty/\mA}$ for the subcategory of cocartesian fibrations relative to $\mE$ and maps of such.

\section{Parametrized $\infty$-categories of presheaves}

Our key tool to prove a local global principle for parametrized $\infty$-categories,
i.e. $\infty$-categories over a given base $\infty$-category $\mC,$ is a parametrized version of presheaf $\infty$-category that we construct in this section.

%associated to any parametrized $\infty$-category.We devote this section to the construction of this parametrized version of presheaf $\infty$-category.

\begin{definition}
A right fibration $\mA \to \mB$ is representable if $\mA$ has a final object.	
\end{definition}
\begin{remark}
A right fibration $\mA \to \mB$ is representable if and only if $\mA \to \mB$ classifies a representable presheaf.
\end{remark}
\begin{notation}
Let $$\mR \subset \Fun([1],\infty\Cat)$$ be the full subcategory of right fibrations and $$\mU \subset \mR$$ the full subcategory of representable right fibrations.
\end{notation}

\begin{lemma}\label{lemmop}
Evaluation at the target $\rho: \mR \to \infty\Cat$ is a cocartesian and cartesian fibration that restricts to a cocartesian fibration $\kappa: \mU \to \infty\Cat$ with the same cocartesian morphisms.
	
\end{lemma}

\begin{proof}
The functor $\gamma: \Fun([1],\infty\Cat) \to \infty\Cat$ evaluating at the target is a cartesian fibration because $\infty\Cat$ admits pullbacks.
Since right fibrations are stable under pullback, $\gamma$ restricts to a cartesian fibration $\rho: \mR \to \infty\Cat$ with the same cartesian morphisms.
For any functor $\psi: \mC \to \mD$ the fiber transport $\mR_\mD \to \mR_\mC$ of $\mR \to \infty\Cat$ admits a left adjoint $\psi_!$ that takes the left Kan extension along the functor $\mC^\op \to \mD^\op.$
%is the functor $\psi^* : \Fun(\mD^\op, \mS) \to \Fun(\mC^\op,\mS)$ precomposing with $\psi.$Since $\psi^*$ admits a left adjoint $\psi_! : \Fun(\mC^\op, \mS) \to \Fun(\mD^\op,\mS)$ for any functor $\psi$, the cartesian fibration $\mR \to \infty\Cat$ is also a cocartesian fibration.Because the functor $\psi_! $ preserves representables for any functor $\psi$, the cocartesian fibration $\mR \to \infty\Cat$ restricts to a cocartesian fibration $\mU \to \infty\Cat$ with the same cocartesian morphsisms.
This guarantees that the cartesian fibration $\rho: \mR \to \infty\Cat$ is also a cocartesian fibration. Since $\psi_!$ preserves representables,
the cocartesian fibration $\rho: \mR \to \infty\Cat$ restricts to a cocartesian fibration $\kappa: \mU \to \infty\Cat$ with the same cocartesian morphisms.
\end{proof}

\begin{definition}\label{parpre}
For every cocartesian fibration $\mD \to \mC$ classifying a functor
$\psi : \mC \to \infty\Cat$ let $\mP^\mC(\mD) \to \mC $ be the pullback of 
$\rho: \mR \to \infty\Cat$ along $\psi: \mC \to \infty\Cat$
\end{definition}

\begin{remark}
Evaluation at the target $\mR \to \infty\Cat$ is a cocartesian and cartesian fibration by Lemma \ref{lemmop}. So for every cocartesian fibration $\mD \to \mC$ 
whose fibers are small, the functor $\mP^\mC(\mD) \to \mC $ is a cocartesian and cartesian fibration.
\end{remark}

\begin{lemma}\label{lqym}
\begin{enumerate}
\item The cocartesian fibration $\kappa: \mU \to \infty\Cat$ classifies the identity.

\vspace{1mm}
By (1) for every cocartesian fibration $\mD \to \mC$ classifying a functor $\mC \to \infty\Cat$ there is a canonical embedding of cocartesian fibrations over $\mC$:
$$\mD \simeq \mC \times_{\infty\Cat} \mU \subset \mP^\mC(\mD)=\mC \times_{\infty\Cat} \mR.$$

\item The embedding $$ \mD \simeq \mC \times_{\infty\Cat} \mU \subset \mP^\mC(\mD)=\mC \times_{\infty\Cat} \mR \subset \mC \times_{\Fun(\{1\},\infty\Cat)} \Fun([1],\infty\Cat)$$ corresponds to a natural transformation of functors $\mD \to \infty\Cat$ that is classified by the canonical map $\mD^{[1]} \to \mD \times_\mC \mD$ of cocartesian fibrations over $\mD.$
\end{enumerate}

\end{lemma}

\begin{proof}
(1): Let $\mV \to \infty\Cat$ be the cocartesian fibration classifying the identity.
We will construct an equivalence $\mV \simeq \mU$
over $\infty\Cat$ by naturally identifying 
for any $\infty$-category $\mA$ the set of equivalence classes of functors $\mA \to \mV$ over $\infty\Cat$ and $\mA \to \mU$ over $\infty\Cat$.

A functor $\mA \to \mV$ over $\infty\Cat$ is classified by a cocartesian fibration
$\mC \to \mA$ equipped with a section of $\mC \to \mA$.
A functor $\mA \to \mU \subset \Fun([1],\infty\Cat)$ over $\infty\Cat$ is classified by a map of cocartesian fibrations $\mD \to \mC$ over $\mA$ that induces on the fiber over any $\A \in \mA$ a representable right fibration.
A section $\alpha$ of $\mC \to \mA$ gives rise to a map $\mA \times_{\mC^{\{1\}} } \mC^{[1]} \to \mC^{\{0\}}$ of cocartesian fibrations over $\mA$ that induces on the fiber over any $\A \in \mA$ the representable right fibration $(\mC_\A)_{/\alpha(\A)} \to \mC_\A$
and induces on sections the functor
$\Fun_\mA(\mA,\mC)_{/\alpha} \to \Fun_\mA(\mA,\mC)$,
which is an equivalence if and only if $\alpha$ is a final object in
$\Fun_\mA(\mA,\mC)$. This is the case if for any $\A \in \mA$ the image $\alpha(\A)$ is final in $\mC_\A.$

Let $\psi: \mD \to \mC$ be a map of cocartesian fibrations over $\mA$ such that for any $\A \in \mA$ the fiber $\mD_\A$ has a final object.
Then the $\infty$-category $\Fun_\mA(\mA,\mD)$
has a final object $\beta$ such that for any $\A \in \mA$ the image $\beta(\A)$ is final in $\mD_\A.$
So we get a a section $\psi \circ \beta$ of $\mC \to \mA$. For $\mD= \mA \times_{\mC^{\{1\}} } \mC^{[1]} $
we have $\Fun_\mA(\mA,\mD) \simeq \Fun_\mA(\mA,\mC)_{/\alpha}$ over $\Fun_\mA(\mA,\mC)$ so that $\psi \circ \beta =\alpha.$
In general the induced map
$\mA \times_{\mD^{\{1\}} } \mD^{[1]} \to \mA \times_{\mC^{\{1\}} } \mC^{[1]} $ of cocartesian fibrations over $\mA$ induces on the fiber over any $\A \in \mA$ the equivalence $(\mD_\A)_{/\beta(\A)} \simeq (\mC_\A)_{/\alpha(\A)}$ (coming from the fact that $\mD_\A \to \mC_\A$ is a right fibration).
Hence the induced map
$\mD \simeq \mA \times_{\mD^{\{1\}} } \mD^{[1]} \to \mA \times_{\mC^{\{1\}} } \mC^{[1]} $ of cocartesian fibrations over $\mA$ is an equivalence.
So $\mD \to \mC$ is equivalent over $\mC$ to $\mA \times_{\mC^{\{1\}} } \mC^{[1]} \to \mC^{\{0\}}$.
The proof of (1) shows (2).

\end{proof}

\begin{definition}Let $\mC \to \mA$ be a functor. The enveloping cocartesian fibration of $\mC \to \mA$ is the pullback 
$$\Env(\mC):= \Fun([1], \mA) \times_{\Fun(\{0\}, \mA)} \mC \to \Fun(\{1\}, \mA).$$ 
\end{definition}
The diagonal embedding $\mA \to \Fun([1],\mA)$ 
induces an embedding $\theta: \mC \to \Env(\mC)$ over $\mA$.
The following proposition is \cite[Proposition A.2.]{HEINE2023108941new}:

\begin{proposition}
Let $\mC \to \mA$ be a functor.
%\begin{remark}The diagonal embedding $\mA \to \Fun([1],\mA)$ is left adjoint to evaluation atthe source. Thus the induced embedding $\theta: \mC \subset \Env(\mC) $ admits a right adjoint relative to $\mC.$\end{remark}
%By \cite[Proposition 2.2.4.9.]{lurie.higheralgebra} 
Restriction along $\theta$ induces for every cocartesian fibration $\mD \to \mA$ an equivalence
$$ \Fun^\cocart_\mA(\Env(\mC),\mD) \to \Fun_\mA(\mC,\mD),$$
where the left hand side is the full subcategory of maps of cocartesian fibrations over $\mA.$	
	
\end{proposition}

\begin{lemma}\label{lemar}
Let $\mC \to \mA$ be a cocartesian fibration between small $\infty$-categories.
There is a canonical embedding $$\mP(\mC) \hookrightarrow \mP^\mA(\Env(\mC))$$ over $\mA$ that induces on the fiber over every $\A \in \mA$ the embedding 
$(\theta_\A)_!: \mP(\mC_\A) \subset \mP(\Env(\mC)_\A)$.
	
\end{lemma}

\begin{proof} By \cite[Lemma A.7.]{HEINE2023108941new} the embedding $\theta: \mC \subset \Env(\mC)$ admits a left adjoint $\L$ relative to $\mA$, which is a map of cocartesian fibrations over $\mA$.
The relative left adjoint gives rise to a map
$\L_!: \mP^\mA(\Env(\mC)) \to \mP^\mA(\mC)$ of cocartesian fibrations over $\mA$.
For every $\A \in \mA$ the induced localization $\L_\A: \Env(\mC)_\A \rightleftarrows \mC_\A: \theta_\A$ induces a localization
$(\L_!)_\A \simeq (\L_\A)_! : \mP(\Env(\mC)_\A)\rightleftarrows \mP(\mC_\A): (\theta_\A)_! $.
By \cite[Proposition 7.3.2.6.]{lurie.higheralgebra} this implies that the map $ \L_!: \mP^\mA(\Env(\mC)) \to \mP^\mA(\mC)$ of cocartesian fibrations over $\mA$ admits a fully faithful right adjoint relative to $\mA$ that induces on the fiber over $\A \in \mA$ the embedding $(\theta_\A)_!.$
%which identifies $\mP^\mA(\mC)$ with the claimed full subcategory of $\mP^\mA(\Env(\mC)) $. % spanned by the objects of $\mP^\mA(\Env(\mC))_\A \simeq \mP(\Env(\mC)_\s) $ that belong to the essential image of the embedding $\mP(\mC_\s) \subset \mP(\Env(\mC)_\s)$ for some $\s \in \mA$.
%that is a map of cartesian fibrations over $\mA$.
	
\end{proof}

Lemma \ref{lemar} motivates the following extension of Definition \ref{parpre}.

\begin{definition}
For every functor $\mC \to \mA$ between small $\infty$-categories let $$\mP^\mA(\mC) \subset \mP^\mA(\Env(\mC)) $$ be the full subcategory spanned by the objects of $\mP^\mA(\Env(\mC))_\A \simeq \mP(\Env(\mC)_\A) $ that belong to the essential image of the embedding $(\theta_\A)_!: \mP(\mC_\A) \subset \mP(\Env(\mC)_\A)$ for some $\A \in \mA$.
\end{definition}

%\begin{remark}For every functor $\mC \to \mA$ whose fibers are small,and every $\A \in \mA$ the induced embedding $\mP^\mA(\mC)_\A \subset \mP^\mA(\Env(\mC))_\A$ is left adjoint to restriction along the embedding $\mC_\A \subset \Env(\mC)_\A.$ \end{remark}
\begin{remark}
The embedding $\mC \subset \Env(\mC) \to \mP^\mA(\Env(\mC))$ over $\mA$ of Lemma \ref{lqym} induces an embedding $\mC \to \mP^\mA(\mC)$ over $\mA$.
\end{remark}
%\begin{lemma}\label{loccc}Let $\phi: \mC \to \mA$ a cocartesian fibration. 
%Let $\phi: \mC \to \mA$ be a locally cocartesian fibration. The embedding $\mC \subset \Env(\mC)$ induces on the fiber over every $\A \in \mA$ a right adjoint embedding $\mC_\A \to \Env(\mC)_\A.$The embedding $\mC \subset \Env(\mC)$ admits a left adjoint relative to $\mA.$\end{lemma}
%\begin{proof}Let $(\X, \f: \phi(\X) \to \Y) \in \Env(\mC)$ and let$\g: \X \to \X'$ be a $\phi$-cocartesian lift of $\f.$ Note that the canonical embedding $\iota: \mC \subset \Env(\mC)$ sends $\Z \in \mC $ to $(\Z,\id_{\phi(\Z)}) \in \Env(\mC). $The morphism $(\g, \id_\Y) : (\X,\f) \to \iota(\X') $ in $\Env(\mC)$ induces for any $\Z \in \mC$ a map$$ \mC(\X', \Z) \to \Env(\mC)((\X, \f), \iota(\Z) ),$$which we will prove to be an equivalence: this map factors as $$ \mC(\X', \Z) \simeq \mA(\Y, \phi(\Z)) \times_{\mA(\phi(\X), \phi(\Z))} \mC(\X, \Z) \simeq  \Env(\mC)((\X, \f), \iota(\Z) ).$$\end{proof}  

\begin{lemma}\label{lemaso}
Let $\phi: \mD \to \mA$ be a locally cocartesian fibration and $\mC \subset \mD$ a full subcategory such that for every $\A \in \mA$ the induced embedding $\mC_\A \subset \mD_\A$ admits a left adjoint $\L_\A.$	
The restriction $\phi': \mC \subset \mD \to \mA$ is a locally cocartesian fibration. A morphism of $\mC$ lying over a morphism $\f: \A \to \B $ of $\mA$ is locally $\phi'$-cocartesian if and only if it factors as $\X \to \Y \to \Y'$, where $\X \to \Y$ is $\phi$-cocartesian and $\Y \to \Y'$ is a $\L_{\B}$-local equivalence.
 
\end{lemma}

\begin{proof}

For every $\Z \in \mC$ lying over $\B \in \mA$ the canonical map
$$ \mC_\B(\Y',\Z) \simeq \mD_\B(\Y,\Z) \to \{\f\}\times_{\mA(\A, \B)} \mD(\X,\Z)$$ is an equivalence.	
	
\end{proof}

\begin{proposition}\label{class}
For every functor $\mC \to \mA$ of small $\infty$-categories the functor $\mP^\mA(\mC) \to \mA$ is a locally cartesian fibration.
%For every morphism $\A \to \B$ of $\mA$ the induced functor $ \mP^\mA(\mC)_\B \to \mP^\mA(\mC)_\A$ admits a right adjoint.
For every morphism $\A \to \B$ of $\mA$ the induced functor $ \mP^\mA(\mC)_\B \to \mP^\mA(\mC)_\A$ factors as $$ \mP(\mC_\B) \subset \mP(\Env(\mC)_\B) \to \mP(\mC_\A),$$
where the last functor is restriction along the functor $\mC_\A \subset \Env(\mC)_\A \to \Env(\mC)_\B,$ and thus is a left adjoint.
\end{proposition}

\begin{proof}
The functor $\mP^\mA(\Env(\mC)) \to \mA $ is a cartesian fibration and for every $\A \in \mA$ the induced embedding $(\theta_\A)_!: \mP^\mA(\mC)_\A \simeq \mP(\mC_\A) \subset \mP^\mA(\Env(\mC))_\A \simeq \mP(\Env(\mC)_\A)$ admits a right adjoint. Thus by Lemma \ref{lemaso} the restriction $\mP^\mA(\mC) \subset \mP^\mA(\Env(\mC)) \to \mA$ is a locally cartesian fibration. 	
\end{proof}

\begin{notation}
Let $\mC \to \mA, \mD \to \mB$ be functors.
Let $$ \Theta(\mC,\mD) := \Fun(\mA,\mB) \times_{\Fun(\mC,\mB)} \Fun(\mC, \mD)$$ and
$$ \Theta(\mC,\mD)^\L \subset \Theta(\mC,\mD)$$ the full subcategory spanned by the commutative squares
$$
\begin{xy}
\xymatrix{
\mC\ar[d]
\ar[rr]
&& \mD \ar[d]^{} 
\\
\mA \ar[rr]^\psi  && \mB}
\end{xy}$$
such that for every $\A \in \mA$ the induced functor
$\mC_\A \to \mD_{\psi(\A)}$ on the fiber over $\A$ preserves small colimits.
	
\end{notation}

The following universal property of $\mP^\mA(\mC) \to \mA$ is crucial for our main theorem:

\begin{proposition}\label{uuuu}Let $\mD \to \mB$ be a locally cartesian fibration whose fibers admit small colimits.
	
\begin{enumerate}
\item Let $\mC \to \mA$ be a functor between small $\infty$-categories.
Restriction along the embedding $\mC \to \mP^\mA(\mC)$ induces an equivalence
$$\theta_\mD: \Theta(\mP^\mA(\mC), \mD)^\L \to \Theta(\mC, \mD).$$

\item Let $\mC \to \mA$ be a functor between $\infty$-categories
and $\mP^\mA(\mC) $ the full subcategory of $\widehat{\mP}^\mA(\mC)$ such that for every
$\A \in \mA$ the fiber $\mP^\A(\mC)_\A \subset \widehat{\mP}^\mA(\mC)_\A \simeq \widehat{\mP}(\mC_\A)$ is generated by $\mC_\A$ under small colimits. 
Restriction along the induced embedding $\mC \to \mP^\mA(\mC) \subset \widehat{\mP}^\mA(\mC)$ induces an equivalence
$$\theta_\mD: \Theta(\mP^\mA(\mC), \mD)^\L \to \Theta(\mC, \mD).$$
	
\end{enumerate}
\end{proposition}

\begin{proof}
(2) is similar to (1). We prove (1).
By the opposite of Lemma \ref{lkj} (applied to a larger universe) there is an embedding $\mD \subset \mE$ over $\mB$ that induces 
on the fiber over every object of $\mA$ a small colimits preserving functor, where $\mE \to \mA$ is a cocartesian and cartesian fibration whose fibers admit large colimits and are locally large.
Since the $\infty$-category of presheaves is generated by the representable presheaves under small colimits \cite[Lemma 2.10.]{heine2024topologicalmodelcellularmotivic}, the functor $\theta_\mD$ is the pullback of the functor $\theta_\mE$.
So it is enough to check that $\theta_\mE$ is an equivalence and we can assume that $\mD \to \mA$ is a cocartesian and cartesian fibration whose fibers admit large colimits and are locally large. 

Since $\mD \to \mB$ is a cocartesian and cartesian fibration, the functor
$\theta_\mD$ is a map of cocartesian and cartesian fibrations over $\Fun(\mA,\mB).$
Consequently, it is enough to see that $\theta_\mD$ induces on the fiber over every functor $\psi: \mA \to \mB$ an equivalence.
The functor $\theta_\mD$ induces on the fiber over $\psi$
the canonical functor $$\Fun_\mA^\L(\mP^\mA(\mC), \mA \times_\mB \mD) \to \Fun_\mA(\mC, \mA \times_\mB \mD),$$ 
where the left hand side is the full subcategory of functors over $\mA$ that induce on every fiber a small colimits preserving functor. 
Therefore it is enough to see that for every cocartesian and cartesian fibration $\mD \to \mA$ whose fibers admit large colimits and are locally large
the canonical functor $$\kappa: \Fun_\mA^\L(\mP^\mA(\mC), \mD) \to \Fun_\mA(\mC, \mD)$$ is an equivalence. 
The functor $\kappa$ is conservative since for every $\A \in \mA$ the $\infty$-category $\mP^\mA(\mC)_\A \simeq \mP(\mC_\A)$ is generated by $\mC_\A$ under small colimits. Thus it is enough to check that $\kappa$ admits a fully faithful left adjoint.
By Proposition \ref{class} the functor $\mP^\mA(\mC) \to \mA$ is a locally cartesian fibration. This implies that $\mP^\mA(\mC)$ is large because
$\mA$ is large and the fibers of $\mP^\mA(\mC) \to \mA$ are large.
Since $\mP^\mA(\mC)$ is large and the fibers of the cocartesian fibration $\mD \to \mA$ admit large colimits preserved by the fiber transports, by \cite[Corollary 4.3.2.14.]{lurie.HTT} and \cite[Proposition 4.3.2.17.]{lurie.HTT} the induced functor $$\Fun_\mA(\mP^\mA(\mC), \mD) \to \Fun_\mA(\mC, \mD)$$ admits a fully faithful left adjoint $\iota$
taking the left Kan extension relative to $\mA$ along the embedding $\mC \to \mP^\mA(\mC).$
We complete the proof by showing that the left Kan extension $\F: \mP^\mA(\mC) \to \mD$ relative to $\mA$ of any functor $\mC \to \mD$ over $\mA$ preserves fiberwise small colimits.
Because $\mP^\mA(\Env(\mC))$ is large, the induced functor $$\Fun_\mA(\mP^\mA(\Env(\mC)), \mD) \to \Fun_\mA(\mP^\mA(\mC), \mD)$$ admits a fully faithful left adjoint $\beta.$
Since the embedding $\mP^\mA(\mC) \subset \mP^\mA(\Env(\mC))$ yields on the fiber over every object of $ \mA$ a left adjoint functor, it is enough to see that $\beta(\F)$ preserves fiberwise small colimits.
Since $\beta(\F) \simeq \beta(\iota(\F_{\mid \mC})) $ is the left Kan extension relative to $\mA$ of
$\F_{\mid \mC}: \mC \to \mD$ along the embedding $\mC \subset \Env(\mC) \to \mP^\mA(\Env(\mC))$ over $\mA,$  it is enough to check that the left Kan extension $\bar{\rH}: \mP^\mA(\Env(\mC)) \to \mD $
relative to $\mA$ of any functor $\rH: \Env(\mC) \to \mD$ over $\mA$ preserves fiberwise small colimits.
Because $\Env(\mC) \subset \mP^\mA(\Env(\mC))$ is a map of cocartesian fibrations over $\mA$, by \cite[Proposition 4.3.3.10.]{lurie.HTT} for every $\A \in \mA$ the induced functor $$\bar{\rH}_\A:  \mP^\mA(\Env(\mC))_\A \simeq \mP(\Env(\mC)_\A) \to \mD_\A $$ is the left Kan extension of $\rH_\A: \Env(\mC)_\A \to \mD_\A$, which by \cite[Lemma 5.1.5.5.]{lurie.HTT} preserves small colimits.

\end{proof}

\begin{lemma}\label{lkj}
	
Let $\mD \to \mA$ be a locally cocartesian fibration of small $\infty$-categories.
There is an embedding $\mD \subset \mE$ over $\mA$ that induces on the fiber over every object of $\mA$ a small limits preserving functor, where $\mE \to \mA$ is a cocartesian and cartesian fibration whose fibers admit small limits and are locally small.
	
\end{lemma}

\begin{proof}
	
Since $\mD \to \mA$ is a locally cocartesian fibration, by \cite[Lemma A.7.]{HEINE2023108941new} the embedding $\mD \subset \Env(\mD)$ induces on the fiber over every object of $\mA$ 
a right adjoint functor.
Since $\mD \to \mA$ is a locally cocartesian fibration of small $\infty$-categories, also $\mD$ is small.
Take $\mE:= \mP^\mA(\Env(\mD)).$ The embedding $\mD \subset \Env(\mD) \to \mP^\mA(\Env(\mD))$ induces on every fiber a small limits preserving functor.
	
\end{proof}

\begin{remark}\label{remre}
Proposition \ref{uuuu} gives for any functors $\mA' \to \mA, \mC \to \mA$ between small $\infty$-categories a canonical functor
\begin{equation}\label{eay}
\mP^{\mA'}( \mA' \times_{\mA} \mC) \to \mA' \times_{ \mA} \mP^{\mA}(\mC)
\end{equation} over $\mA',$
which is a map of locally cartesian fibrations over $\mA'$, and induces equivalences on fibers and so is an equivalence.
\end{remark}

\begin{lemma}\label{class1}
	
Let $\mC \to \mA$ be a locally cartesian fibration between small $\infty$-categories.
The embedding $\mC \subset \mP^\mA(\mC) $ is a map of locally cartesian fibrations over $\mA.$
The functor $\mC \to \mA$ is a cartesian fibration if and only if the functor
$ \mP^\mA(\mC) \to \mA$ is a cartesian fibration.
	
\end{lemma}

\begin{proof}
We first prove that the embedding $\mC \subset \mP^\mA(\mC) $ is a map of locally cartesian fibrations over $\mA$ if $\mC \to \mA$ is a locally cartesian fibration.
By Remark \ref{remre} for every functor $\mA' \to \mA$ there is a canonical equivalence
$$\mA' \times_{ \mA} \mP^{\mA}(\mC) \simeq \mP^{\mA'}( \mA' \times_{\mA} \mC)
$$ over $\mA'$ so that we can reduce to the case $\mA=[1].$
Let $\mM \to [1]$ be a functor and $\mC:=\mM_0,\mD:=\mM_1$ the fibers.
The functor $\mP(\mM)^{[1]} \to [1]$ is a cartesian fibration classifying the 
functor $\phi: \mP(\mD) \subset \mP(\mM) \to \mP(\mC)$,
where the first functor takes left Kan extension along the embedding $\mD \subset \mM$ and the second functor restricts along the embedding $\mC \subset \mM.$
If $\mM\to [1]$ is a cartesian fibration classifying a functor
$\G: \mD \to \mC,$ then $\phi$ is the extension of the functor
$\mD \xrightarrow{\G} \mC \subset \mP(\mC).$
So the embedding $\mM \subset \mP(\mM)^{[1]}$ is a map of cartesian fibrations over $[1].$
	
A locally cartesian fibration $\phi$ is a cartesian fibration if and only if
the collection of $\phi$-cartesian morphisms is closed under composition.
Hence $\mC \to \mA$ is a cartesian fibration if the functor
$ \mP^\mA(\mC) \to \mA$ is a cartesian fibration.
Conversely, if $\mC \to \mA$ is a cartesian fibration, the functor $ \mP^\mA(\mC) \to \mA$ is a cartesian fibration since the fibers of 
$\mP^\mA(\mC) \to \mA$ are generated by the representable presheaves under small colimits, and the fiber transports of $\mP^\mA(\mC) \to \mA$ preserve small colimits.
	
\end{proof}

\begin{lemma}\label{locol}
Let $\mC \to \mD$ be a map  of locally (co)cartesian fibrations over $\mA$.
The functor $\mP^\mA(\mC) \to \mP^\mA(\mD)$
is a map of locally (co)cartesian fibrations over $\mA$.
\end{lemma}
\begin{proof}By Lemma \ref{class1} for any locally cartesian fibration
$\mC \to \mA$ the functor $\mP^\mA(\mC) \to \mA$ is a locally cartesian fibration and the embedding $\mC \to \mP^\mA(\mC) $ is a map of locally cartesian fibrations over $\mA$.
By Lemma \ref{lemar} and Remark \ref{remre} for every locally cocartesian fibration
$\mC \to \mA$ the functor $\mP^\mA(\mC) \to \mA$ is a locally cocartesian and locally cartesian fibration and the embedding $\mC \to \mP^\mA(\mC) $ is a map of locally cocartesian fibrations over $\mA$.
Since the fibers of $\mP^\mA(\mC) \to \mA$ are generated under small colimits by the representables and the fiber transports of the locally (co)cartesian fibration $\mP^\mA(\mC) \to \mA$ preserve small colimits by Proposition \ref{class}, the induced functor $\mP^\mA(\mC) \to \mP^\mA(\mD)$
is a map of locally (co)cartesian fibrations over $\mA$.
\end{proof}

\begin{definition}
A functor $\phi: \mD \to \mC$ of small $\infty$-categories
is an exponential fibration if the functor $\mD \times_\mC (-): \infty\Cat_{/\mC} \to \infty\Cat_{/\mC}$ taking pullback along $\phi$ preserves small colimits or equivalently admits a right adjoint.
	
\end{definition}

\begin{notation}
	
For every exponential fibration $\phi: \mD \to \mC$ let $\Fun^\mC(\mD,-): \infty\Cat_{/\mC} \to \infty\Cat_{/\mC}$ be the right adjoint of the functor $\mD \times_\mC (-): \infty\Cat_{/\mC} \to \infty\Cat_{/\mC}$.
	
\end{notation}

The next remark is \cite[Remark 3.71.]{HEINE2023108941}:

\begin{remark}
	
The pullback of any exponential fibration $\mD \to \mC$ along any functor $\mB \to \mC$ is an exponential fibration and for every functor $\mE \to \mC$ there is a canonical equivalence $\mB \times_\mC 	\Fun^\mC(\mD,\mE) \simeq \Fun^\mB(\mB \times_\mC \mD,\mB \times_\mC \mE)$ over $\mB.$
	
\end{remark}

\begin{proposition}\label{expo}
A functor $\mC \to \mA$ of small $\infty$-categories
is an exponential fibration if and only if the locally cartesian fibration $\mP^{\mA}(\mC) \to \mA$ is a cartesian fibration.
\end{proposition}

\begin{proof}
A functor $\mB \to \mA$ is an exponential fibration (cartesian fibration) if and only if the pullbacks along any functors $[2] \to \mA$ are exponential fibrations (cartesian fibrations).
So by Remark \ref{remre} we can reduce to the case that $\mA=[2].$
The locally cartesian fibration $\mP^{[2]}(\mC) \to [2]$ classifies 
three functors $\f: \mP(\mC_2) \to \mP(\mC_1), \g: \mP(\mC_1) \to \mP(\mC_0),  \h: \mP(\mC_2) \to \mP(\mC_0)$ and a natural transformation
$\alpha: \g \circ \f \to \h.$ 
By \cite[Proposition B.3.2.]{lurie.higheralgebra} a functor $\mC \to [2]$ is an exponential fibration if and only if 
$\alpha$ is an equivalence.
By \cite[Proposition 2.4.2.8.]{lurie.HTT} the locally cartesian fibration $\mP^{\mA}(\mC) \to \mA$ is a cartesian fibration if and only if 
$\alpha$ is an equivalence.
	
\end{proof}

\begin{corollary}Every functor $\mD \to [1]$ is an exponential fibration.
So the $\infty$-category $\infty\Cat_{/[1]}$ is cartesian closed.
\end{corollary}

\begin{corollary}
	
Let $\phi:\mA \to \mB$ be an inclusion of $\infty$-categories
such that for every morphisms $\X \to \Y, \Y \to \Z$ we have that $\Y \in \mA$ if $\X,\Z \in \mA.$ Then $\phi$ is an exponential fibration.
	
\end{corollary}

\begin{proof}
	
By Proposition \ref{expo} the functor $\phi$ is an exponential fibration if and only if the locally cartesian fibration $\mP^{\mB}(\mA) \to \mB$ is a cartesian fibration.
By \cite[Proposition 2.4.2.8.]{lurie.HTT} the locally cartesian fibration $\mP^{\mB}(\mA) \to \mB$ is a cartesian fibration if and only if for every morphisms $\alpha: \X \to \Y, \beta: \Y \to \Z $ the induced natural transformation $\theta: \alpha^* \circ \beta^* \to (\beta\circ \alpha)^*: \mP(\mA_\Z) \to \mP(\mA_\X)$ is an equivalence.
If $\X$ does not belong to $\mA$, the fiber of $\phi$ over $\X$ is empty so that
$\mP(\mA_\X)$ is contractible. So there is nothing to show and we assume that $\X $ belongs to $\mA$ so that $\mP(\mA_\X)\simeq \mS.$
If $\Z$ does not belong to $\mA$, the fiber of $\phi$ over $\Z$ is empty so that
$\mP(\mA_\Z)$ is contractible. In this case there is nothing to show, too, because $\alpha^*,\beta^*$ preserve small colimits and so the initial object so that $\theta$ is an endomorphism of the initial object in $\mP(\mA_\X)\simeq \mS$ and so an equivalence.
So we can assume that $\Z$ belongs to $\mA$. 
By assumption we find that $\Y$ belongs to $\mA$.
In this case the small colimits preserving functors $\alpha^*,\beta^*:\mS \to \mS$ are the identity on objects and so preserve the final object. This implies that $\alpha^*,\beta^*:\mS \to \mS$ are equivalences and $\theta$ is an equivalence since the component of $\theta$ at the final space is an endomorphism of the final space and so an equivalence.	
	
%Let $\phi: [\n]\to [\m]$ be a map of the form $[\n] \cong \{\bi+0,...,\bi+\n \} \subset [\m]$ for some $\bi \geq 0$.By Proposition \ref{expo} the functor $\phi$ is an exponential fibration if and only ifthe locally cartesian fibration $\mP^{[\m]}([\n]) \to [\n]$ is a cartesian fibration.By \cite[Proposition 2.4.2.8.]{lurie.HTT} the locally cartesian fibration $\mP^{[\m]}([\n]) \to [\n]$ is a cartesian fibration if and only if for every $0 \leq \bk \leq \ell \leq \br \leq \m$the induced natural transformation $\alpha$ from $\mP^{[\m]}([\n])_\br \to \mP^{[\m]}([\n])_\ell \to \mP^{[\m]}([\n])_\bk $ to $\mP^{[\m]}([\n])_\br \to \mP^{[\m]}([\n])_\bk$ is an equivalence.If $\bk < \bi$, there is nothing to show since $\mP^{[\m]}([\n])_\bk \simeq *$ since the fiber of $\phi$ over $\bk$ is empty. So we can assume that $\bk \leq \bi.$If $\br >\bi+\n $, then $\mP^{[\m]}([\n])_\br \simeq *$ since the fiber of $\phi$ over $\br$ is empty. In this case So there is nothing to show.So we can assume that $\br \leq \bi+\n$. Moreover we can assume that $\bk \leq \bi$ because otherwise  $\mP^{[\m]}([\n])_\br \simeq \mS$The functors The natural transformation $\alpha$
	
\end{proof}

\begin{definition}\label{inert}
A map $[\n]\to [\m]$ of $\Delta $ is inert if it is of the form	
$[\n] \cong \{\bi+0,...,\bi+\n \} \subset [\m]$ for some $\bi \geq 0$.
	
\end{definition}

\begin{corollary}\label{ine}
	
Every inert map $[\n] \to [\m]$ of $\Delta$ (viewed as a functor of $\infty$-categories) is an exponential fibration.	
	
\end{corollary}

\begin{corollary}\label{reco}
A locally cartesian fibration of small $\infty$-categories
is a cartesian fibration if and only if it is an exponential fibration.

\end{corollary}

\begin{proof}
	
Let $\mC \to \mA$ be a locally cartesian fibration of small $\infty$-categories.
By Lemma \ref{class1} the functor $\mC \to \mA$ is a cartesian fibration if and only if the functor
$ \mP^\mA(\mC) \to \mA$ is a cartesian fibration.
By Proposition \ref{expo} the functor $\mC \to \mA$ is an exponential fibration if and only if the functor $\mP^{\mA}(\mC) \to \mA$ is a cartesian fibration.
	
\end{proof}

%\begin{construction}\label{const}Let $\mB$ be an $\infty$-category and $\mC \to \mA \times \mB$a map of cocartesian fibrations over $\mB$ classifying a functor $\gamma: \mB \to \infty\Cat_{/\mA}$.The functor $\mP^{\mA\times\mB}(\mC) \to \mA \times \mB$ is a map of cocartesian fibrations over $\mB$ classifying a functor $\mB \to \widehat{\Cat}_{\infty/\mA}.$For $\gamma$ the identitywe obtain a functor $\mP^\mA: \infty\Cat_{/\mA} \to \widehat{\Cat}_{\infty/\mA}$ sending $\mC$ to $\mP^\mA(\mC).$The embedding $\mC \subset \mP^{\mA\times\mB}(\mC)$ is a map ofcocartesian fibrations over $\mB$ over $\mA \times \mB$ classifying for $\gamma$ the identity a map $\id \to \mP^\mA$ of functors $\infty\Cat_{/\mA} \to \widehat{\Cat}_{\infty/\mA}$ whose component at a functor $\mD \to \mA$ is the canonical embedding $\mD \to \mP^\mA(\mD).$\end{construction}

\begin{notation}\label{notate}
Let $$\mathrm{LoCART} \subset \Fun([1],\infty\Cat)$$ be the full subcategory of locally cartesian fibrations.
\end{notation}
\begin{notation}Let $$\mathrm{LoCART}^\L \subset \mathrm{Lo}\widehat{\CART}$$ be the subcategory of locally cartesian fibrations whose target is small and whose fibers are presentable and commutative squares 
$$
\begin{xy}
\xymatrix{
\mC\ar[d]
\ar[rr]
&& \mD \ar[d]^{} 
\\
\mA \ar[rr]^\psi && \mB}
\end{xy}$$
such that for every $\A \in \mA$ the induced functor
$\mC_\A \to \mD_{\psi(\A)}$ on the fiber over $\A$ preserves small colimits.

%that induce fiber-wise a small colimits preserving functor. 
%intersection of the subcategories$\mathrm{LocCart}$ and $\Fun([1],\infty\widehat{\Cat})^{\rc\rc}.$
\end{notation}

For the next corollary we use the notion of adjunction of $(\infty,2)$-categories, i.e.
$\infty$-categories enriched in $\infty\Cat.$
This is a special case of the notion of enriched adjunction (see \cite[Definition 2.66.]{heine2024bienriched} or \cite[6.2.6]{heinerestricted} for details).

\begin{corollary}

The inclusion of $(\infty,2)$-categories $$ \mathrm{LoCART}^\L \subset \Fun([1],\infty\widehat{\Cat})$$
admits a left adjoint that sends a functor $\mC \to \mA$ between small $\infty$-categories to $\mP^\mA(\mC)\to \mA.$
%In particular, the 2-functor $\infty\Cat^\op \to \Enr\Fun_{\infty\Cat}(\mathrm{LoCART}^\L, \infty\widehat{\Cat}), \mC \mapsto \Fun(\mC,-)$induces a 2-functor $\infty\Cat^\op \to (\mathrm{LoCART}^\L)^\op$ opposite to a 2-functor $\mP: \infty\Cat \to \mathrm{LoCART}^\L.$
	
\end{corollary}

\begin{corollary}Let $\mA$ be a small $\infty$-category. The inclusion of $(\infty,2)$-categories $$ \mathrm{LoCART}^\L_\mA \subset \infty\widehat{\Cat}_{/\mA}$$
admits a left adjoint that sends a functor $\mC \to \mA$, where $\mC$ is small, to $\mP^\mA(\mC)\to \mA.$
	
\end{corollary}

In the following we compare Definition \ref{parpre} to another one
(Corollary \ref{homint}).

\begin{notation}
Let $\mL \subset \Fun([1],\infty\Cat)$ be the full subcategory of left fibrations.	
	
\end{notation}

\begin{proposition}
	
There is a canonical equivalence
$$ \mL \simeq \Fun^{\infty\Cat}(\mU,\mS \times \infty\Cat)$$
over $\infty\Cat,$	
where $\mL \to \infty\Cat$ is evaluation at the target.
	
\end{proposition}

\begin{proof}Let $\mC \to \infty\Cat$ be a functor classifying a cocartesian fibration $\mB \to \mC.$
We prove that for any $\infty$-category $\mC$ functors
$\mC \to \mL$ over $\infty\Cat$ naturally correspond to functors
$\mC \to \Fun^{\infty\Cat}(\mU,\mS \times \infty\Cat)$ over $\infty\Cat$.
Functors $\mC \to \Fun^{\infty\Cat}(\mU,\mS \times \mC)$ over $\infty\Cat$ correspond to sections of the functor $\Fun^{\mC}(\mB,\mS \times \mC) \to \mC$.
Such sections correspond to functors $\mB \to \mS \times \mC$ over $\mC$ and so to functors
$\mB \to \mS.$
A functor $\mB \to \mL \subset \Fun([1],\infty\Cat)$ over $\infty\Cat$ is classified by a map of cocartesian fibrations $\mA \to \mB$ over $\mC$ that is fiberwise a left fibration and so a left fibration by \cite[Proposition 2.4.2.4., Proposition 2.4.2.11.]{lurie.HTT}.
Such a functor corresponds to a left fibration $\mA \to \mB$ classifying a functor $\mB \to \mS.$
	
\end{proof}

\begin{corollary}
	
There is a canonical equivalence
$$ \mR \simeq \Fun^{\infty\Cat}(\mU^\rev,\mS \times \infty\Cat)$$
over $\infty\Cat,$	
where $\mR \to \infty\Cat$ is evaluation at the target.

\end{corollary}

\begin{corollary}\label{homint} Let $\mC \to \mA$ be a cocartesian fibration. There is a canonical equivalence
$$ \mP^\mA(\mC) \simeq \Fun^{\mA}(\mC^\rev,\mS \times \mA)$$ over $\mA.$	

\end{corollary}

\begin{remark}
In \cite[Example 9.9.]{barwick2016parametrizedhighercategorytheory} the authors give a different construction
of a parametrized $\infty$-category of presheaves associated to a cocartesian fibration.
Corollary \ref{homint} implies that our construction of parametrized $\infty$-category of presheaves associated to a functor $\phi: \mC \to \mA$ agrees with the one of \cite{barwick2016parametrizedhighercategorytheory} in case that $\phi$ is a cocartesian fibration.
	
\end{remark}

\section{A local-global principle}

Now we are ready to prove the main theorem.

\begin{theorem}\label{corr}Let $\mA$ be an $\infty$-category.
The canonical commutative square of $\infty$-categories
\begin{equation}\label{sqkh}
\begin{xy}
\xymatrix{
\infty\Cat_{/\mA} \ar[d]
\ar[rr]^{\mP^\mA}
&& \mathrm{LoCART}^\L_\mA \ar[d]^{} 
\\
\infty\Cat_{/\iota(\mA)} \ar[rr]^{\mP^{\iota(\mA)}}  && \mathrm{LoCART}^\L_{\iota(\mA)}}
\end{xy} 
\end{equation}
is a pullback square.

\end{theorem}

\begin{proof}
%So it is enough to prove that square (\ref{sqkh}) is a pullback square. 

Let $\mC \to \mA,\mD \to \mA$ be functors of small $\infty$-categories.
Square (\ref{sqkh}) induces on morphism $\infty$-categories the commutative square
\begin{equation*}\label{sqkk}
\begin{xy}
\xymatrix{
\Fun_\mA(\mC,\mD) \ar[d]
\ar[rr]
&& \Fun^\L_\mA(\mP^\mA(\mC), \mP^\mA(\mD)) \ar[d]^{} 
\\
\Fun_{\iota(\mA)}(\iota(\mA) \times_\mA \mC, \iota(\mA) \times_\mA \mD) \ar[rr]^{}  && \Fun^\L_{\iota(\mA)}(\iota(\mA) \times_\mA \mP^\mA(\mC), \iota(\mA) \times_\mA \mP^\mA(\mD)),}
\end{xy} 
\end{equation*}
which identifies with the commutative square
\begin{equation}\label{sqkkk}
\begin{xy}
\xymatrix{
\Fun_\mA(\mC,\mD) \ar[d]
\ar[rr]
&& \Fun_\mA(\mC, \mP^\mA(\mD)) \ar[d]^{} 
\\
\Fun_{\iota(\mA)}(\iota(\mA) \times_\mA \mC, \iota(\mA) \times_\mA \mD) \ar[rr]^{}  && \Fun_{\iota(\mA)}(\iota(\mA) \times_\mA \mC, \iota(\mA) \times_\mA \mP^\mA(\mD)),}
\end{xy} 
\end{equation}
where we use Remark \ref{remre}.
The bottom and top functor of square (\ref{sqkkk}) are fully faithful,
which guarantees that square (\ref{sqkkk}) is a pullback square.
This implies that the functor $$\theta: \infty\Cat_{/\mA} \to \infty\Cat_{/\iota(\mA)} \times_{\mathrm{LoCART}^\L_{\iota(\mA)}} \mathrm{LoCART}^\L_{\mA} $$ 
is fully faithful.
So it remains to see that $\theta$ is essentially surjective.

Let $\mB \to \iota(\mA)$ be a functor, $ \mD \to \mA$ a locally cartesian fibration
whose fibers admit small colimits and whose fiber transports preserve small colimits, and $\alpha: \mP^{\iota(\mA)}(\mB) \simeq \iota(\mA) \times_\mA \mD $ an equivalence over $\iota(\mA)$.
Let $\mC \subset \mD$ be the essential image of the Yoneda-embedding
$\mB \subset \mP^{\iota(\mA)}(\mB) \simeq \iota(\mA) \times_\mA \mD \subset \mD$.
Then there is a canonical equivalence $\iota(\mA) \times_\mA\mC \simeq \mB$ over $\iota(\mA)$ and the embedding $\mC \subset \mD$ over $\mA$ extends to a canonical functor $\beta: \mP^\mA(\mC) \to \mD$ over $\mA,$
whose pullback to $\iota(\mA)$ is the equivalence
$\alpha: \iota(\mA) \times_\mA\mP^\mA(\mC) \simeq \mP^{\iota(\mA)}(\mB) \simeq \iota(\mA) \times_\mA\mD.$
Therefore it is enough to see that $\beta$ is an equivalence.
For this %it is enough to check that $\beta$ induces an equivalence after pulling back along any functor $[1]\to \mA$.Consequently, 
by Remark \ref{remre} we can reduce to the case that $\mA=[1]$
because a functor over $\mA$ is an equivalence if all pullbacks along functors $[1] \to \mA$ are equivalences. Since $\beta$ is essentially surjective, we need to check that for any $\X \in \mP(\mC_0), \Y \in \mP(\mC_1)$ the canonical map $\theta_{\X,\Y}: \mP^{[1]}(\mC)(\X, \Y) \to \mD(\beta(\X),\beta(\Y))$ induced by $\beta$ is an equivalence.
Let $\G: \mP(\mC_1) \to \mP(\mC_0)$ be the left adjoint functor classified by the cartesian fibration $\mP^{[1]}(\mC) \to [1]$ and $\G': \mD_1 \to \mD_0$ the left adjoint functor classified by the cartesian fibration $\mD \to [1]$.
The unique cartesian lift $\G(\Y) \to \Y$, where $\G(\Y) \in \mP(\mC_0)$,
is sent to a morphism $\beta(\G(\Y)) \to \beta(\Y)$
that factors as $\beta(\G(\Y)) \xrightarrow{\rho_\Y} \G'(\beta(\Y)) \to \beta(\Y)$, where $\G'(\beta(\Y)) \in \mD_0$. The map $\theta_{\X,\Y}$ identifies with the map
$$ \mP(\mC_0)(\X, \G(\Y)) \to \mD_0(\beta(\X), \beta(\G(\Y))) \to \mD_0(\beta(\X), \G'(\beta(\Y))).$$
Since the first map is an equivalence, it is enough to see that $\rho_\Y$ is an equivalence for every $\Y \in \mP(\mC_1).$
Since $\beta_\mC$ is fully faithful, $\theta_{\X,\Y}$ is an equivalence if
$\X \in \mC_0, \Y \in \mC_1$.
Because $\mP(\mC_0)$ is generated by $\mC_0$ under small colimits,
$\theta_{\X,\Y}$ is an equivalence for all $\X \in \mP(\mC_0), \Y \in \mC_1$.
Thus $\rho_\Y$ is an equivalence for every $\Y \in \mC_1$.
Therefore $\rho_\Y$ is also an equivalence for every $\Y \in \mP(\mC_1)$
because $\mP(\mC_1)$ is generated by $\mC_1$ under small colimits and the functors $\G, \G'$ preserve small colimits and $\rho_\Y$ is natural in $\Y$.
\end{proof}

\begin{corollary}\label{alfo}
	
The canonical commutative square
\begin{equation}\label{sqkt}
\begin{xy}
\xymatrix{
\Fun([1],\infty\Cat) \ar[d]
\ar[rr]^{\mP}
&&\mathrm{LoCART}^\L \ar[d]^{} 
\\
\iota^*(\Fun([1],\infty\Cat)) \ar[rr]^{\iota^*(\mP)}  && \iota^*(\mathrm{LoCART}^\L)}
\end{xy} 
\end{equation}
of cartesian fibrations over $\infty\Cat$ is a pullback square. 	
	
\end{corollary}

\begin{proof}
	
Square (\ref{sqkt}) induces on the fiber over every small $\infty$-category $\mA$
the pullback square (\ref{sqkh}) of $\infty$-categories.	
	
\end{proof}

\begin{remark}
	
We expect that the analogue of Theorem \ref{corr} also holds for $(\infty,n)$-categories for
$ 1 \leq n \leq \infty$ for appropriate notions of locally cartesian fibrations.
Here the role of the parametrized category of presheaves is taken by a parametrized version of
presheaves of higher categories that can be conveniently described via the language of enriched presheaves \cite{HINICH2020107129}, \cite{heine2024equivalencemodelsinftycategoriesenriched}.
This will be topic of future work.

\end{remark}

\begin{notation}
Let 
\begin{itemize}
\item $\mathrm{EXP}\subset \Fun([1],\infty\Cat)$
be the full subcategory of exponential fibrations,
\item $\mathrm{CART} \subset \Fun([1],\infty\Cat) $ be the full subcategory of cartesian fibrations,
\item $\CART^\L \subset \mathrm{LoCART}^\L $ be the full subcategory of cartesian fibrations. 
\end{itemize}	
	
\end{notation}

Corollary \ref{alfo} and Proposition \ref{expo} imply the following corollary:
\begin{corollary}\label{corflat}
The commutative square
\begin{equation}\label{sqk}
\begin{xy}
\xymatrix{
\mathrm{EXP} \ar[d]
\ar[rr]^{\mP}
&& \CART^\L \ar[d]^{} 
\\
\iota^*(\mathrm{EXP}) \ar[rr]^{\iota^*(\mP)}  && \iota^*(\CART^\L)}
\end{xy} 
\end{equation}
of cartesian fibrations over $\infty\Cat$ is a pullback square. 
\end{corollary}

%\begin{proof}%By Proposition \ref{expo} a functor $\mM \to \mA$ of small $\infty$-categories is an exponential fibration if and only if the functor $\mP^{\mA}(\mM) \to \mA$ is a cartesian fibration.So it remains to see that a functor $\theta: \mM \to \mM'$ over $\mA$is strict if and only if the induced functor$\mP^\mA(\mM) \to \mP^\mA(\mM')$ preserves cartesian morphisms over $\mA$.By Remark \ref{remre} we can reduce to the case $\mA=[1].$
%By Proposition \ref{class} the locally cartesian fibration $\mP^{[1]}(\mM) \to [1]$ classifies the functor $\psi:\mP(\mM_1) \to \mP(\mM) \to \mP(\mM_0)$, where the first functor takes left Kan extension along the embedding $\mM_1 \subset \mM$ and the second functor restricts along the embedding $\mM_0 \subset \mM$, and similar for $\mM' \to [1].$ The functor $\gamma: \mP^{[1]}(\mM) \to \mP^{[1]}(\mM')$ over $[1]$ classifies a natural transformation $\rho: \mP(\theta_0) \circ \psi \to \psi' \circ \mP(\theta_1).$ The functor $\theta$over $[1]$ is strict if $\rho$ is an equivalence, which is equivalent to say that $\gamma$ is a map of cartesian fibrations over $[1]$.
%preserves cartesian morphisms over $[1].$	\end{proof}

\section{Correspondences as left fibrations}

The next definition is \cite[Notation 3.3.]{heine2024infinitycategoriesdualityhermitian}:

\begin{definition}Let $\mC$ be an $\infty$-category.
The twisted arrow left fibration $\Tw(\mC)\to \mC^\op \times \mC$ is the left fibration  classifying the mapping space functor $\mC^\op \times \mC \to \mS.$	
	
\end{definition}

%By \cite[Notation 3.3.]{heine2024infinitycategoriesdualityhermitian} there is a functor $\Tw:\infty\Cat \to \infty\Cat$ and natural transformations$\Tw \to (-)^\op, \Tw \to \id$ of functors $\infty\Cat \to \infty\Cat$ whose component at any small $\infty$-category $\mC$ give the twisted arrow left fibration $\Tw(\mC)\to \mC^\op \times \mC.$

\begin{theorem}\label{lfib} There is a canonical pullback square of $\infty$-categories
$$
\begin{xy}
\xymatrix{
\infty\Cat_{/[1]}
\ar[d]
\ar[rr]
&&  \mL \ar[d]^{} 
\\
\infty\Cat \times \infty\Cat \ar[rr]^{(-)^\op\times(-)}  && \infty\Cat,}
\end{xy}$$
where the right hand vertical functor evaluates at the target.
The induced equivalence from the left upper corner of the square to the pullback %of the right vertical and bottom horizontal functor 
sends a functor $\mM \to [1]$ to
%$\mM_0, \mM_1$ and the left fibration over $ \mM_0^\op \times \mM_1$ classifying the functor $\mM_0^\op \times \mM_1 \subset \mM^\op \times \mM \xrightarrow{\mM(-,-)}\mS.$
the triple $$(\mM_0, \mM_1, \mM_0^\op \times \mM_1 \times_{\mM^\op \times \mM} \Tw(\mM) \to  \mM_0^\op \times \mM_1).$$

\end{theorem} 

\begin{proof}
	
By Theorem \ref{corr} there is a canonical pullback square of $\infty$-categories:
$$
\begin{xy}
\xymatrix{
\infty\Cat_{/[1]}
\ar[d]
\ar[rr]
&& \mathrm{CART}^\L_{[1]} \ar[d]^{} 
\\
\infty\Cat \times \infty\Cat \ar[rr]^{\mP\times\mP}  && \Pr^\L \times \Pr^\L.}
\end{xy}$$
%Sending an $\infty$-category to its opposite $\infty$-category induces an equivalence$\mR \simeq \mL$.
Consequently, it is enough to construct an equivalence between the pullbacks
$$
\begin{xy}
\xymatrix{
\mA
\ar[d]
\ar[rr]
&&  \mL \ar[d]^{} 
\\
\infty\Cat \times \infty\Cat \ar[rr]^{(-)^\op\times(-)}  && \infty\Cat}
\end{xy}$$
and 
$$
\begin{xy}
\xymatrix{
\mB
\ar[d]
\ar[rr]
&& \mathrm{CART}^\L_{[1]}  \ar[d]^{} 
\\
\infty\Cat \times \infty\Cat \ar[rr]^{\mP\times\mP}  && \Pr^\L \times \Pr^\L.}
\end{xy}$$
We construct this equivalence by giving for every $\infty$-category $\mA$ a bijection between equivalence classes of objects of the $\infty$-categories $ \Fun(\mA,\mA), \Fun(\mA,\mB)$ natural in $\mA.$
A functor $\mA \to \mB$ is classified by a pair of cocartesian fibrations
$\mX \to \mA, \mY \to \mA$ and a map of cocartesian fibrations $\mM \to \mA \times [1]$ over $\mA$ that induces on the fiber over every object $\A\in \mA$ a cartesian fibration $\mM_\A \to [1]$
classifying a small colimits preserving functor $\mM_\A^1 \to \mM_\A^0$ and such that there are equivalences $\mM_0 \simeq \mP^\mA(\mX), \mM_1 \simeq \mP^\mA(\mY)$ over $\mA$.
By \cite[Lemma 2.44.]{heine2023monadicity} a functor $\mW \to \mA \times [1]$ is a map of cocartesian fibrations over $\mA$ that induces on the fiber over every object $\A\in \mA$ a cartesian fibration over $ [1]$ if and only if it is a map of cartesian fibrations over $[1]$ that induces on the fiber over every object of $[1]$ a cocartesian fibration over $\mA$.
Consequently, a functor $\mA \to \mB$ is likewise classified by a pair of cocartesian fibrations
$\mX \to \mA, \mY \to \mA$ and a map of cartesian fibrations $\mM \to \mA \times [1]$ over $[1]$ that induces on the fibers over $0,1 \in [1]$ the cocartesian fibrations $\mP^\mA(\mX), \mP^\mA(\mY)$, respectively, and induces on the fiber over every object $\A\in \mA$ a cartesian fibration $\mM_\A \to [1]$ classifying a small colimits preserving functor $\mM_\A^1 \to \mM_\A^0$.
In other words, a functor $\mA \to \mB$ is classified by a pair of cocartesian fibrations
$\mX \to \mA, \mY \to \mA$ and a functor $\mP^\mA(\mY) \to \mP^\mA(\mX)$ over $\mA$ that induces on the fiber over every object of $\mA$ a small colimits preserving functor.
By Proposition \ref{uuuu} the latter is uniquely determined by its restriction, a functor $\mY \to \mP^\mA(\mX)$ over $\mA$, which corresponds by Proposition \ref{homint} to a functor $\mX^\rev \times_\mA \mY \to \mS$.
The latter is classified by a left fibration $\mA \to \mX^\rev \times_\mA \mY$
or equivalently by a map of cocartesian fibrations $\mA \to \mX^\rev \times_\mA \mY $ over $\mA$
that induces on the fiber over every object of $\mA$ a left fibration.
So a functor $\mA \to \mB$ is classified by a pair of cocartesian fibrations
$\mX \to \mA, \mY \to \mA$ and a map of cocartesian fibrations $\mA \to \mX^\rev \times_\mA \mY $ over $\mA$ that induces on the fiber over every object of $\mA$ a left fibration.
The latter is precisely classified by a functor $\mA \to \mA.$
	
\end{proof}

\begin{remark}
Under the equivalence \begin{equation}\label{vvbt}
\infty\Cat_{/[1]} \simeq (\infty\Cat \times \infty\Cat)\times_{\infty\Cat} \mL\end{equation} of Theorem \ref{lfib} the opposite $\infty$-category involution on $\infty\Cat_{/[1]}$ corresponds to an involution on the right hand pullback
induced by the flip action on the product $\infty\Cat \times \infty\Cat.$
This gives an easy way to understand the opposite $\infty$-category involution on $\infty\Cat_{/[1]}$ and is used in \cite[Theorem 5.21.]{heine2024infinitycategoriesdualityhermitian} and \cite[Theorem 6.27.]{heine2021realktheorywaldhauseninfinity} to compare different models of dualities.
In fact in \cite[Proposition 5.23.]{heine2024infinitycategoriesdualityhermitian} equivalence (\ref{vvbt}) is enhanced to a $\C_2$-equivariant equivalence for the $C_2$-actions refining the opposite $\infty$-category involution and the flip action.
	
\end{remark}

\begin{notation}
	
Let $\Cocart \subset \Fun([1],\infty\Cat)$ be the full subcategory of cocartesian fibrations.	
	
\end{notation}

\begin{corollary}\label{fib} Let $\mC,\mD$ be small $\infty$-categories. There is a canonical equivalence
$$\{(\mC,\mD)\}\times_{\infty\Cat \times \infty\Cat} \Cat_{/[1]} \simeq \Fun(\mC^\op, \Fun(\mD,\mS))$$
that sends a functor $\mM \to [1]$ and $\mM_0 \simeq \mC, \mM_1 \simeq \mD$
to the functor $\mC^\op \times \mD \subset \mM^\op \times \mM \xrightarrow{\mM(-,-)} \mS.$ %classified by the left fibration $\mC^\op \times \mD \times_{\mM^\op \times \mM} \Tw(\mM) \to \mC^\op \times \mD. $

The latter equivalence restricts to an equivalence
$$\{(\mC,\mD)\}\times_{\infty\Cat \times \infty\Cat} \Cocart_{[1]} \simeq \Fun(\mC^\op, \mD^\op) \simeq \Fun(\mC,\mD)^\op$$
that sends a functor $\mM \to [1]$ and $\mM_0 \simeq \mC, \mM_1 \simeq \mD$
to the functor $\mC \to \mD$ classified by $\mM \to [1]$.

\end{corollary} 

\begin{proof}By Theorem \ref{lfib} there is a canonical equivalence
$$\{(\mC,\mD)\}\times_{\infty\Cat \times \infty\Cat} \Cat_{/[1]} \simeq \{\mC^\op \times \mD \}\times_{\infty\Cat} \mL \simeq \Fun(\mC^\op\times\mD,\mS) \simeq \Fun(\mC^\op, \Fun(\mD,\mS))$$
that sends a functor $\mM \to [1]$ and $\mM_0 \simeq \mC, \mM_1 \simeq \mD$
to the functor $\mC^\op \times \mD \subset \mM^\op \times \mM \xrightarrow{\mM(-,-)} \mS.$ %classified by the left fibration $\mC^\op \times \mD \times_{\mM^\op \times \mM} \Tw(\mM) \to \mC^\op \times \mD. $

The latter equivalence restricts to an equivalence
$$\{(\mC,\mD)\}\times_{\infty\Cat \times \infty\Cat} \Cocart_{[1]} \simeq \Fun(\mC^\op, \mD^\op) \simeq \Fun(\mC,\mD)^\op$$
that sends a functor $\mM \to [1]$ and $\mM_0 \simeq \mC, \mM_1 \simeq \mD$
to the functor $\mC \to \mD$ classified by $\mM \to [1]$.	
	
For every cocartesian fibration $\mM \to [1]$ and $\mM_0 \simeq \mC, \mM_1 \simeq \mD$
classifying a functor $\alpha: \mC \to \mD$ there is a canonical equivalence
between the functor $\mC^\op \times \mD \subset \mM^\op \times \mM \xrightarrow{\mM(-,-)} \mS$
and the functor $\mC^\op \times \mD \xrightarrow{\alpha^\op \times \mD} \mD^\op \times \mD \xrightarrow{\mD(-,-)} \mS$.
	
\end{proof}

\section{An end formula for mapping spaces of parametrized functor $\infty$-categories}

In this section we describe the mapping space functor of the internal hom $\Fun^{[1]}(\mM,\mN)$ between two functors $\mM \to [1],\mN \to [1]$ (Theorem \ref{mapinner}).
Having this description we relate the mapping spaces of the conditionally existing internal hom $\Fun^{\mC}(\mM,\mN)$ between two functors $\mM \to \mC,\mN \to \mC$ to the mapping space functors of $\mM,\mN$
(Corollary \ref{cmapin}).
As a corollary we compute the mapping space functor of the $\infty$-category
$\Fun(\mA,\mB)$ of functors between two $\infty$-categories $\mA,\mB$
(Corollary \ref{endd}) recovering a global version of a result of \cite[Proposition 5.1.]{articles}.

\begin{lemma}
	
Let $\mV$ be a closed monoidal $\infty$-category, $\mW \subset \Fun([1],\mV)$ a full monoidal subcategory such that the base change of any object $\X \to \Y$ of $\W$ along any morphism $\Z \to \X$ of $\mV$ exists and is an object of $\mW$ and for every object $\X \to \Y$ of $\mW$ and every object $\Z$ of $\mV$ the induced morphism $\Mor_\mV(\Z,\X) \to \Mor_\mV(\Z,\Y)$ in $\mV$ is an object of $\mW.$
Let $\alpha: \mA \to \mV$ be a monoidal functor. The pullback $\mA \times_{\mV} \mW$ of evaluation at the target $\mW \to \mV$ along $\alpha$, which is a monoidal $\infty$-category, is closed.
The internal hom of every $(\A, \f: \X \to \alpha(\A)), (\B, \g: \Y \to \alpha(\B))\in \mA \times_{\mV} \mW $ is $$(\Mor_\mA(\A,\B), \alpha(\Mor_\mA(\A,\B)) \times_{\Mor_\mV(\X, \alpha(\B))} \Mor_\mV(\X,\Y) \to \alpha(\Mor_\mA(\A,\B))) \in \mA \times_{\mV} \mW.$$
	
\end{lemma}

\begin{proof}
For every $(\A, \f: \X \to \alpha(\A)), (\B, \g: \Y \to \alpha(\B))\in \mA \times_{\mV} \mW, (\C, \h: \Z \to \alpha(\C))\in \mA \times_{\mV} \mW $
the canonical morphism $$(\Mor_\mA(\A,\B), \alpha(\Mor_\mA(\A,\B)) \times_{\Mor_\mV(\X, \alpha(\B))} \Mor_\mV(\X,\Y) \ot (\A, \f: \X \to \alpha(\A)) =$$$$ (\Mor_\mA(\A,\B)\ot\A, \alpha(\Mor_\mA(\A,\B)\ot \A) \times_{\Mor_\mV(\X, \alpha(\B))\ot \X} \Mor_\mV(\X,\Y) \ot \X) \to (\B, \g: \Y \to \alpha(\B))$$
in $\mA \times_{\mV} \mW$ induces a map
$$ \mA \times_{\mV} \mW((\A, \f: \X \to \alpha(\A)), $$$$(\Mor_\mA(\B,\C), \alpha(\Mor_\mA(\B,\C)) \times_{\Mor_\mV(\Y, \alpha(\C))} \Mor_\mV(\Y,\Z) \to \alpha(\Mor_\mA(\B,\C))) ) \simeq$$$$
\mA \times_{\mV} \mW((\A, \f: \X \to \alpha(\A)) \ot (\B, \g: \Y \to \alpha(\B)),(\C, \h: \Z \to \alpha(\C)))=$$$$ \mA \times_{\mV} \mW((\A \ot \B, \f \ot \g: \X \ot \Y \to \alpha(\A \ot \B)),(\C, \h: \Z \to \alpha(\C)))$$ 
over $\mA(\A, \Mor_\mA(\B,\C)) \simeq \mA(\A \ot \B,\C)$ that induces on the fiber over every morphism $\theta: \A \to \Mor_\mA(\B,\C)$ corresponding to a morphism $\theta': \A \ot \B \to \C$ the canonical equivalence
$$ \mV_{/\Mor_\mV(\Y, \alpha(\C))}(\X, \Mor_\mV(\Y,\Z)) \simeq \mV_{/\alpha(\C)}(\X\ot\Y,\Z).$$ 
	
\end{proof} 

\begin{corollary}\label{homm}
The pullback $$
\begin{xy}
\xymatrix{
\mQ	\ar[d]
\ar[rr]
&&  \mL \ar[d]^{} 
\\
\infty\Cat \times \infty\Cat \ar[rr]^{(-)^\op\times(-)}  && \infty\Cat}
\end{xy}$$	
in cartesian symmetric monoidal $\infty$-categories is closed.
The internal hom of every $(\mA,\mB, \mC, \mC \to \mA^\op \times \mB), (\mA',\mB', \mC', \mC' \to \mA'^\op \times \mB')$ is $$(\Fun(\mA,\mA'),\Fun(\mB,\mB'), \Fun(\mA,\mA')^\op \times \Fun(\mB,\mB') \times_{\Fun(\mC, \mA'^\op \times \mB')} \Fun(\mC, \mC')$$$$ \to \Fun(\mA,\mA')^\op \times \Fun(\mB,\mB')).$$
	
\end{corollary}

\begin{notation}
For every functor $\mM \to [1]$ let $\widetilde{\Tw}(\mM)$ be the pullback $$ \mM_0^\op \times \mM_1 \times_{\mM^\op\times\mM}\Tw(\mM) \to \mM_0^\op \times \mM_1.$$
	
\end{notation}

\begin{remark}\label{ihol}
Let $\mC$ be an $\infty$-category. Projection $\mC \times [1]\to \mC$ induces an equivalence $\widetilde{\Tw}(\mC\times [1]) \simeq \Tw(\mC)$ over $ \mC^\op \times \mC.$
		
\end{remark}

%There is a canonical equivalence$$\Fun(\mM_0,\mN_0)^\op \times \Fun(\mM_1,\mN_1) \times_{\Fun^{[1]}(\mM,\mN)^\op \times \Fun^{[1]}(\mM,\mN)} \Tw(\Fun^{[1]}(\mM,\mN)) \simeq $$$$\Fun(\mM_0,\mN_0)^\op \times \Fun(\mM_1,\mN_1) \times_{\Fun(\mM_0^\op \times \mM_1 \times_{\mM^\op \times \mM} \Tw(\mM), \mN_0^\op \times \mN_1)} $$$$ \Fun(\mM_0^\op \times \mM_1 \times_{\mM^\op \times \mM} \Tw(\mM), \mN_0^\op \times \mN_1 \times_{\mN^\op \times \mN} \Tw(\mN))$$

%Theorem \ref{lfib} and Corollary \ref{homm} give the following:

\begin{theorem}\label{mapinner}
Let $\mM \to [1], \mN \to [1]$ be functors. There is a canonical equivalence
$$\widetilde{\Tw}(\Fun^{[1]}(\mM,\mN)) \simeq $$
$$\Fun(\mM_0,\mN_0)^\op \times \Fun(\mM_1,\mN_1) \times_{\Fun( \widetilde{\Tw}(\mM), \mN_0^\op \times \mN_1)} \Fun( \widetilde{\Tw}(\mM), \widetilde{\Tw}(\mN))$$
over $$ \Fun(\mM_0,\mN_0)^\op \times \Fun(\mM_1,\mN_1)$$
that induces on the fiber over any functors $\F: \mM_0 \to \mN_0, \G: \mM_1 \to \mN_1$ an equivalence
$$\Fun^{[1]}(\mM,\mN)(\F,\G) \simeq \lim(\widetilde{\Tw}(\mM) \to \mM_0^\op \times \mM_1 \xrightarrow{\F^\op \times \G} \mN_0^\op \times \mN_1 \subset \mN^\op \times \mN \xrightarrow{\mN(-,-)} \mS).$$	
	
\end{theorem}

\begin{proof}
	
By Theorem \ref{lfib} there is a canonical equivalence $\theta$ between 
$\infty\Cat_{/[1]}$ and the pullback 
$$\begin{xy}
\xymatrix{
\mQ	\ar[d]
\ar[rr]
&&  \mL \ar[d]^{} 
\\
\infty\Cat \times \infty\Cat \ar[rr]^{(-)^\op\times(-)}  && \infty\Cat.}
\end{xy}$$	
Consequently, the internal hom of $\infty\Cat_{/[1]}$ corresponds under $\theta$ to the
internal hom of $\mQ$ for the cartesian monoidal structures.
The description of $\theta$ of Theorem \ref{lfib} and the description of the internal hom
of Corollary \ref{homm} imply the result.

\end{proof}

\begin{corollary}\label{cmapin}
Let $\mA \to \mC,\mB \to \mC$ be functors, $\f: \X \to \Y$ a morphism in $\mC$ and $\F: \mA_\X \to \mB_\X, \G: \mA_\Y \to \mB_\Y$ functors. There is a canonical equivalence
$$\{\alpha\}\times_{\mC(\X,\Y)} \Fun^\mC(\mA,\mB)(\F,\G) \simeq $$$$ \lim(\widetilde{\Tw}([1]\times_ \mC \mA) \to \mA_\X^\op \times \mA_\Y \xrightarrow{\F^\op \times \G} \mB_\X^\op \times \mB_\Y \to ([1]\times_ \mC \mB)^\op \times ([1]\times_ \mC \mB) \xrightarrow{([1]\times_ \mC \mB)(-,-)} \mS).$$
	
\end{corollary}

\begin{proof}
There is a canonical equivalence
$$\{\alpha\}\times_{\mC(\X,\Y)} \Fun^\mC(\mA,\mB)(\F,\G) \simeq ([1]\times_ \mC \Fun^\mC(\mA,\mB))(\F,\G) \simeq$$$$ \Fun^{[1]}([1]\times_ \mC \mA,[1]\times_ \mC \mB)(\F,\G) \simeq $$$$ \lim(\widetilde{\Tw}([1]\times_ \mC \mA) \to \mA_\X^\op \times \mA_\Y \xrightarrow{\F^\op \times \G} \mB_\X^\op \times \mB_\Y \to ([1]\times_ \mC \mB)^\op \times ([1]\times_ \mC \mB) \xrightarrow{([1]\times_ \mC \mB)(-,-)} \mS).$$	
	
\end{proof}

\begin{definition}
Let $\mC,\mD$ be $\infty$-categories and $\phi: \mC^\op \times \mC \to \mD$ a functor.
The end of $\phi$, denoted by ${\int \phi}$ if it exists, is the limit of the functor $\Tw(\mC) \xrightarrow{\q} \mC^\op \times \mC \xrightarrow{\phi} \mD.$
	
\end{definition}

\begin{remark}Let $\phi: \mC^\op \times \mC \to \mS$ be a functor classified by a left fibration
$\mB \to \mC^\op \times \mC.$
	
The limit of the functor $\Tw(\mC) \xrightarrow{\q} \mC^\op \times \mC \xrightarrow{\phi} \mS$
is the space $$\Fun_{\Tw(\mC)}(\Tw(\mC), \q^*(\mB)).$$
	
\end{remark}

Theorem \ref{mapinner} implies the following corollary in view of Remark \ref{ihol}:

\begin{corollary}\label{endd}
Let $\mC,\mD$ be $\infty$-categories. There is a canonical equivalence
$$\Tw(\Fun(\mC,\mD)) \simeq \Fun(\mC,\mD)^\op \times \Fun(\mC,\mD) \times_{\Fun(\Tw(\mC), \mD^\op \times \mD)} \Fun(\Tw(\mC), \Tw(\mD))$$
over $$ \Fun(\mC,\mD)^\op \times \Fun(\mC,\mD).$$	
The latter induces on the fiber over every functors $\F,\G:\mC \to \mD$ an equivalence
$$ \Fun(\mC,\mD)(\F,\G) \simeq \Fun_{\Tw(\mC)}(\Tw(\mC),\q^*((\F^\op\times\G)^*(\Tw(\mD))))\simeq \int\mD(\F(-),\G(-)).$$
	
\end{corollary}

\section{Correspondences as profunctors}

In the following we use the notion of double $\infty$-category \cite[Definition 2.4.3.]{MR3345192}, which is an $\infty$-category internal to $\infty\Cat:$

\begin{definition}
A double $\infty$-category is a cartesian fibration $\mB \to \Delta$
%^\op \to \infty\Cat$ is a Segal $\infty$-category 
such that for every $[\n] \in \Delta$
%the functor $\F$ sends the colimitof the diagram $$[1] \coprod_{[0]} [1] ... \coprod_{[0]} [1] \
%simplicial $\infty$-category $\F:\Delta^\op \to \infty\Cat$ is a Segal $\infty$-category if for every $[\n] \in \Delta$
%the functor $\F$ sends the colimitof the diagram $$[1] \coprod_{[0]} [1] ... \coprod_{[0]} [1] \simeq \{0,1\} \coprod_{\{1\}} \{1,2\} ...\coprod_{\{\n-1\}} \{\n-1,\n\} \simeq [\n]$$to a limit diagram.
the induced functor 
%$$\F([\n]) \to \F([1]) \times_{\F([0])}... \times_{\F([0])}\F([1])$$
$$\mB_{[\n]} \to \mB_{[1]} \times_{\mB_{[0]}}... \times_{\mB_{[0]}}\mB_{[1]}$$
is an equivalence, where the $\bi$-th map $[1] \to [\n]$ sends $0$ to $\bi-1$ and
$1$ to $\bi.$

\end{definition}

\begin{example}

The functor $$\rho: \infty\Cat^\cart_{/\Delta} \to \infty\Cat$$ taking the fiber over $[0]$
admits a fully faithful right adjoint that sends an $\infty$-category $\mC$ to the cartesian fibration $\Delta_\mC \to \Delta$ classifying the functor $$\Delta^\op \subset \infty\Cat^\op \xrightarrow{\Fun(-,\mC)} \infty\Cat.$$
The cartesian fibration $\Delta_\mC \to \Delta$ is a double $\infty$-category.

%whose fiber over any $[\n]$ is the product $

\end{example}

\begin{remark}

The functor $\rho: \infty\Cat^\cart_{/\Delta} \to \infty\Cat$ is a cartesian fibration, where a morphism
$\mB \to \mC$ over $\Delta$ is $\rho$-cartesian if and only if the following square is a pullback square:
$$
\begin{xy}
\xymatrix{
\mB
\ar[d]
\ar[rr]
&& \mC \ar[d]^{} 
\\
\Delta_{\mB_{[0]}} \ar[rr]  && \Delta_{\mC_{[0]}} .}
\end{xy}$$
This holds since for every cartesian fibration $\mA \to \Delta$ the induced commutative square
$$
\begin{xy}
\xymatrix{
\infty\Cat^\cart_{/\Delta}(\mA,\mB)
\ar[d]
\ar[rr]
&&\infty\Cat^\cart_{/\Delta}(\mA,\mC) \ar[d]^{} 
\\
\infty\Cat^\cart_{/\Delta}(\mA,\Delta_{\mB_{[0]}}) \ar[rr]  && \infty\Cat^\cart_{/\Delta}(\mA,\Delta_{\mC_{[0]}})}
\end{xy}$$
is equivalent to the commutative square
$$
\begin{xy}
\xymatrix{
\infty\Cat^\cart_{/\Delta}(\mA,\mB)
\ar[d]
\ar[rr]
&&\infty\Cat^\cart_{/\Delta}(\mA,\mC) \ar[d]^{} 
\\
\infty\Cat(\mA_{[0]},\mB_{[0]}) \ar[rr]  && \infty\Cat(\mA_{[0]},\mC_{[0]}).}
\end{xy}$$

In particular, the $\rho$-cartesian lift to a double $\infty$-category $\mC \to \Delta$
of a functor $ \mA \to \mC_{[0]}$ is the double $\infty$-category $ \Delta_\mA \times_{\Delta_{\mC_{[0]}}} \mC.$

\end{remark}

\begin{notation}
Let $\mA \to \Delta, \mB \to \Delta	$ be double $\infty$-categories.
We write $$\FUN(\mA,\mB) \subset \Fun_{\Delta}(\mA,\mB)$$ for the full subcategory spanned by the functors over $\Delta$ preserving cocartesian lifts of morphisms of $\Delta.$	
	
\end{notation}

\begin{definition}\label{pree}

An $(\infty,2)$-precategory is a double $\infty$-category $\phi: \mB \to \Delta$ 
such that $\mB_0$ is a space.
% and the composition $\Delta^\op \to \infty\Cat \xrightarrow{\iota}\mS$ of the functor classified by $\phi$ and the maximal subspace functor $\iota$ is a complete Segal space.	

\end{definition}

\begin{remark}
By \cite[5.6.1. Corollary]{HINICH2020107129} there is an equivalence between $(\infty,2)$-precategories in the sense of Definition \ref{pree} and $\infty$-precategories enriched in $\infty\Cat$, which justifies our terminology.
%Moreover there is an equivalence between $(\infty,2)$-categories and $\infty$-precategories enriched in $\infty\Cat$, which justifies our terminology.
\end{remark}

\begin{example}
For every double $\infty$-category $\phi: \mB \to \Delta$ the pullback
$$ {\Delta_{\iota(\mB_{[0]})} \times_{\Delta_{\mB_{[0]}}} \mB} \to \Delta_{\iota(\mB_{[0]})} $$
is an $(\infty,2)$-precategory, which we call the underlying $(\infty,2)$-precategory of $\phi.$
	
\end{example}

%\begin{remark}There is a canonical equivalence $\iota(\FUN(\N,\mC)) \simeq \infty\Cat(-,\mC')$so that the functor $\iota(\FUN(\N,\mC)): \infty\Cat^\op \to \mS$ preserves small limits.Hence also the functor $\infty\Cat(\K, \FUN(\N,\mC)) \simeq \iota(\Fun(\K,\FUN(\N,\mC))) \simeq \iota(\FUN(\N,\mC^\K)): \infty\Cat^\op \to \infty\Cat$ preserves small limits so that also the functor $\FUN(\N,\mC): \infty\Cat^\op \to \infty\Cat$ preserves small limits.\end{remark}

\begin{example}
Viewing spaces as $\infty$-categories any Segal space classifies an $(\infty,2)$-precategory.	
For every $\infty$-category $\mC$ the nerve of $\mC$ defined by
$\N(\mC): \Delta^\op \to \mS, [\n]\mapsto \infty\Cat([\n],\mC)$ is a Segal space and so classifies an $(\infty,2)$-precategory.	
\end{example}

For the next notation let $\tau:=(-)^\op: \infty\Cat \to \infty\Cat$ be the opposite $\infty$-category involution.

\begin{notation}\label{notis}
Let	
\begin{enumerate}
\item $\infty\CAT:= \tau^*(\Delta \times_{\infty\Cat} \mathrm{CART}) \to \Delta,$
\item $\mathrm{PR}^\L:= \tau^*(\Delta \times_{\infty\widehat{\Cat}} \mathrm{CART}^\L) \to \Delta,$
\item $\mathrm{CORR}:= \tau^*(\Delta \times_{\infty\Cat} \mathrm{EXP}) \to \Delta,$
\item $\mathrm{CORR}_\Lax:= \tau^*(\Delta \times_{\infty\Cat} \Fun([1],\infty\Cat)) \to \Delta.$
%\item $\widetilde{\mathrm{CAT}}_\infty:=\Delta \times_{\infty\Cat} \Fun([1],\infty\Cat) \to \Delta.$
\end{enumerate}

\end{notation}

\begin{construction}
Let $\rH: \K \to \infty\Cat$ be a functor that admits a colimit $\mA$.
The canonical $\infty\Cat$-linear functor $\Cat_{\infty /\mA} \to \lim_{\bk \in \K}\Cat_{\infty /\rH(\bk)}$
induces on morphism $\infty$-categories between two functors $\mC \to \mA,\mD \to \mA$ a functor
\begin{equation}\label{eee}
\theta: \Fun_{\mA}(\mC,\mD) \to \lim_{\bk \in \K} \Fun_{\rH(\bk)}(\rH(\bk) \times_{\mA} \mC,\rH(\bk) \times_{\mA}\mD).
\end{equation}
\end{construction}

\begin{lemma}\label{oho}
Let $\rH: \K \to \infty\Cat$ be a functor	with colimit $\mA$ and $\mC \to \mA,\mD \to \mA$ functors.
If the functor $\mC \to \mA$ is an exponential fibration, the functor
(\ref{eee}) is an equivalence.

\end{lemma}

\begin{proof}
Since $\mC \to \mA$ is an exponential fibration,
the functor $(-) \times_{\mA}\mC : \Cat_{\infty /\mA} \to \Cat_{\infty /\mC}
\to  \Cat_{\infty /\mA}$ preserves small colimits.
Hence the canonical functor $\colim_{\bk \in \K}(\rH(\bk) \times_{\mA} \mC) \to \mC $ over $ \colim (\rH) \simeq \mA $ is an equivalence.
Since the functor $\Fun_{\mA}(-,\mD): \Cat_{\infty /\mA}^\op \to \infty\Cat$ preserves limits, the following canonical functor is an equivalence:
$$\Fun_{\mA}(\mC,\mD) \to \Fun_{\mA}(\colim_{\bk \in \K}\rH(\bk) \times_{\mA} \mC,\mD) \simeq \lim_{\bk \in \K} \Fun_{\rH(\bk)}(\rH(\bk) \times_{\mA} \mC,\rH(\bk) \times_{\mA}\mD).$$

\end{proof}

%\begin{proposition}Let $[\n] \in \Delta$ and $\mC \to [\n]$ a functor and$$\mC' := (\{0 \to 1\} \times_{[\n]} \mC) \times_{(\{1\} \times_{[\n]} \mC)} (\{1 \to 2\} \times_{[\n]} \mC) \times_{(\{2\} \times_{[\n]} \mC)} ... \times_{(\{\n-1\} \times_{[\n]} \mC)}(\{\n-1 \to \n\} \times_{[\n]} \mC).$$The canonical functor $\mC' \to \mC$ over $[\n]$ induces for every exponential fibration $\mD \to [\n]$ an equivalence$$\Fun_{[\n]}(\mD, \mC') \to \Fun_{[\n]}(\mD, \begin).$$ \end{proposition}

%\begin{lemma}Let $\rH: \K \to \infty\Cat$ be a functor	with colimit $\mA$.The canonical functor of $(\infty,2)$-categories$\CART_\mA \to \lim_{\bk \in \K}\CART_{\rH(\bk)}$is an equivalence.\end{lemma}

%Lemma \ref{oho} implies the following corollary:

For the next corollary we use that for every $\n \geq 0$ the category $[\n]$ is the colimit in $\infty\Cat$ of the canonical diagram
\begin{equation}\label{diags}
\{0<1\}\coprod_{\{1\}} \{1<2\} \coprod_{\{2\}} ... \coprod_{\{\n-1\}} \{\n-1<\n\},\end{equation}
in which all maps are inert. Lemma \ref{oho} gives the following corollary:

\begin{corollary}\label{hok}
Let $\n \geq 0$ and $\mB \to \mC $ a functor over $[\n]$ whose pullback along any inert map $[1]\to [\n]$ is an equivalence. For every exponential fibration $\mD \to [\n]$ the induced functor $$\Fun_{[\n]}(\mD,\mB) \to \Fun_{[\n]}(\mD,\mC)$$ is an equivalence.	
	
\end{corollary}

\begin{notation}Let $\n \geq 0$ and $\phi: \mC \to [\n]$ a functor.
Pulling back diagram (\ref{diags}) along $\phi$ we obtain a diagram 
$$ \{0<1\} \times_{[\n]} \mC \times_{(\{1\} \times_{[\n]} \mC)} \{1<2\} \times_{[\n]} \mC \times_{(\{2\} \times_{[\n]} \mC)} ... \times_{(\{\n-1\} \times_{[\n]} \mC)} \{\n-1<\n\} \times_{[\n]} \mC $$ over $\mC$.
We write $\mC'$ for the colimit of the latter diagram. %, which we see over $[\n]$via $\mC' \to \mC \to [\n].$		
	
\end{notation}

\begin{remark}\label{hato}
By 	Corollary \ref{ine} the inert maps $[1]\cong \{\bi-1< \bi\}\subset [\n]$ for any $\bi \leq 0$
are exponential fibrations.
Hence for every $\bi \leq 0$ the pullback of the canonical functor $\mC' \to \mC$ over $[\n]$
along the map $[1]\cong \{\bi-1< \bi\}\subset [\n]$ identifies with the equivalence
$$ \emptyset \times_{(\{\bi-1\} \times_{[\n]} \mC)}\{\bi-1<\bi\}\times_{[\n]} \mC \times_{(\{\bi\} \times_{[\n]} \mC)} \emptyset \to \{\bi-1<\bi\}\times_{[\n]} \mC.$$
	
\end{remark}

Corollary \ref{hok} and Remark \ref{hato} give the following:

\begin{corollary}Let $\n \geq 0$ and $\mC \to [\n]$ a functor.
%Let $$\mC':= \colim(\{0<1\} \times_{[\n]} \mC \times_{(\{1\} \times_{[\n]} \mC)} \{1<2\} \times_{[\n]} \mC \times_{(\{2\} \times_{[\n]} \mC)} ... \times_{(\{\n-1\} \times_{[\n]} \mC)} \{\n-1<\n\} \times_{[\n]} \mC).$$
For every exponential fibration $\mD \to [\n]$ the induced functor
$$\Fun_{[\n]}(\mD,\mC') \to \Fun_{[\n]}(\mD,\mC)$$ is an equivalence.
	
\end{corollary}

%\begin{proof}The induced functor$\Fun_{[\n]}(\mD,\mC') \to \Fun_{[\n]}(\mD,\mC)$ factors asthe induced functor$$\Fun_{[\n]}(\mD,\mC') \to \Fun_{[\n]}(\mD,\mC)$$ and so is an equivalence.\end{proof}

We obtain the following proposition:

\begin{proposition}\label{jjkl} Let $\n \geq 0$. The embedding $\mathrm{EXP}_{[\n]} \subset \infty\Cat_{/[\n]}$ admits a right adjoint. A functor $\mC \to \mD$ over $[\n]$ is a colocal equivalence if and only if it induces an equivalence after pulling back along any inert map $[1] \to [\n].$
	
\end{proposition}

\begin{proposition}\label{oho2}
	
The cartesian fibration $\CART \to \infty\Cat$ classifies the functor
$$ \FUN(\tau^*\N,\infty\CAT): \infty\Cat^\op \to \infty\widehat{\Cat}.$$
	
%There is a canonical equivalence $$ \FUN(\N(\mA),\infty\CAT) \simeq \CART_\mA$$natural in $\mA$.	
%that is the right Kan extension of its restriction to $\Delta.$ 
\end{proposition}

\begin{proof}
By Yoneda the cartesian fibration $\Delta \times_{\infty\Cat} \Cart=\tau^*\infty\CAT \to \Delta$ classifies the functor
$$ \FUN(\N_{\mid \Delta},\tau^*\infty\CAT) \simeq \FUN(\tau^*\N_{\mid \Delta},\infty\CAT): \Delta^\op \to \infty\widehat{\Cat}.$$
Since the full subcategory $\Delta \subset \infty\Cat$ is dense, it is enough to verify that the functor $ \infty\Cat^\op \to \infty\widehat{\Cat}$ classified by the cartesian fibration $\CART \to \infty\Cat$ and the functor $ \FUN(\tau^*\N,\infty\CAT): \infty\Cat^\op \to \infty\widehat{\Cat}$ send for every $\infty$-category $\mC$ the colimit of the functor $\rho: \Delta \times_{\infty\Cat} \infty\Cat_{/\mC} \to \Delta \subset \infty\Cat$ to a limit.
We first prove the second statement: %it suffices to show that the functor $ \FUN(\N,\infty\CAT): \infty\Cat^\op \to \infty\widehat{\Cat}$ preserves small limits. For this it is enough to see that 
because the functor $\FUN(-,\infty\Cat): \mathrm{Seg}(\infty\Cat)^\op \to \infty\Cat$ preserves small limits, it suffices to show that the embedding $\N: \infty\Cat \to \mathrm{Seg}(\mS)\subset \mathrm{Seg}(\infty\Cat)$ preserves the colimit of the functor $\Delta \times_{\infty\Cat} \infty\Cat_{/\mC} \to \Delta \subset \infty\Cat$. 
The embedding $\N: \infty\Cat \to \mathrm{Seg}(\mS)$ preserves this colimit because $\mathrm{Seg}(\mS)$ is a localization of $\mP(\Delta)$ and so for every $\infty$-category $\mC$ the Segal space $\N(\mC) $ is the colimit of the functor $\Delta \times_{\mathrm{Seg}(\mS)} \mathrm{Seg}(\mS)_{/\N(\mC)} \to \Delta \subset \infty\Cat \subset \mathrm{Seg}(\mS)$, which factors as $\N \rho.$
The embedding $ \mathrm{Seg}(\mS) \subset \mathrm{Seg}(\infty\Cat)$
preserves small colimits since the embedding %$\mS \subset \infty\Cat$ and so the embedding
$\Fun(\Delta^\op,\mS) \subset \Fun(\Delta^\op,\infty\Cat)$ preserve small colimits and Segal equivalences:
the generating Segal equivalences for $\Fun(\Delta^\op,\mS)$ are the canonical morphisms
$ \Delta^1 \coprod_{\Delta^0} ...  \coprod_{\Delta^0} \Delta^1 \to \Delta^\n$ for $\n \geq 0$,
while the generating Segal equivalences for $\Fun(\Delta^\op,\infty\Cat)$ are the canonical morphisms $ \Delta^1 \coprod_{\Delta^0} ...  \coprod_{\Delta^0} \Delta^1 \times \mA \to \Delta^\n \times \mA$ for $\n \geq 0$ and $\mA$ a small $\infty$-category.

We finish the proof by showing that the cartesian fibration $\CART \to \infty\Cat$ classifies a small limits preserving functor $ \infty\Cat^\op \to \infty\widehat{\Cat}$.
Let $\rH: \K \to \infty\Cat$ be a functor that admits a colimit $\mA$.
The canonical $\infty\Cat$-linear functor $\rho: \CART_{\mA} \to \lim_{\bk \in \K}\CART_{\rH(\bk)}$
induces on morphism $\infty$-categories between two functors $\mC \to \mA,\mD \to \mA$ the functor (\ref{eee}). The latter is an equivalence by Lemma \ref{oho}
because by Corollary \ref{reco} every cartesian fibration is an exponential fibration.
The induced map $\iota(\rho)$ is essentially surjective because it identifies with the canonical equivalence $$\iota(\Fun(\mA,\infty\Cat^\op)) \to \lim_{\bk \in \K}\iota(\Fun(\rH(\bk),\infty\Cat^\op)).$$

\end{proof}

%\begin{corollary}\label{oho3}The cartesian fibration $\infty\CAT \to \Delta$ is a double $\infty$-category.\end{corollary}

%\begin{corollary}\label{oho3}The cartesian fibration $\infty\CAT \to \Delta$ is a double $\infty$-categoryclassifying a functor $\Delta^\op \to \infty\widehat{\Cat}$ whose right Kan-extension to$\infty\Cat^\op$ classifies $\CART \to \infty\Cat$.

%for every small $\infty$-category $\mA$\end{corollary}

%\begin{proof}

%A functor $\infty\Cat^\op \to \infty\widehat{\Cat}$ is the right Kan-extension of its restriction to $\Delta^\op$ if and only if it preserves small limits since $\Delta$ is dense in $\infty\Cat$ by ...\end{proof}

%Since exponential fibrations are stable under pullback, the cartesian fibration $\Fun([1],\infty\Cat) \to \infty\Cat$ evaluating at the target restricts to a cartesian fibration $\mathrm{CORR}\to \infty\Cat$with the same cartesian morphisms.

\begin{corollary}\label{oho3}
	
The cartesian fibration $\CART^\L \to \infty\Cat$ classifies the functor
$$ \FUN(\tau^*\N,\PR^\L): \infty\Cat^\op \to \infty\widehat{\Cat}.$$
	
\end{corollary}

Corollary \ref{corflat} implies the following theorem:

\begin{theorem}\label{corre} The cartesian fibrations (1)-(3) of Notation \ref{notis} %$$\infty\CAT \to \Delta, \ \mathrm{PR}^\L\to \Delta, \CORR \to \Delta$$ 
are double $\infty$-categories and there is a canonical pullback square of double $\infty$-categories:
$$
\begin{xy}
\xymatrix{
\mathrm{CORR}
\ar[d]
\ar[rr]
&& \mathrm{PR}^\L \ar[d]^{} 
\\
\Delta_{\infty\Cat} \ar[rr]  && \Delta_{\Pr^\L}.}
\end{xy}$$
%of cartesian fibrations over $\Delta$ classifying Segal $\infty$-categories. 0583401238

\end{theorem}

%\section{A generalized Grothendieck-construction}

For the next proposition we view $\infty\Cat$ as $(\infty,2)$-category via its internal hom:

\begin{proposition}\label{enri} %0041583401880

The $(\infty,2)$-category underlying the double $\infty$-category $\infty\CAT$
is $\infty\Cat$ viewed as enriched in itself via the internal hom.	

\end{proposition} %+41583401238,  +49 22895300 gkv spitzenverband deutsche verbindungsstelle

\begin{proof}

Let $\infty\CAT'$ be the $(\infty,2)$-category underlying the double $\infty$-category $\infty\CAT$.
Then the space of objects of $\infty\CAT'$ is $\iota(\infty\CAT_{[0]}) \simeq \iota(\infty\Cat)$
and for every $\A,\B \in \infty\Cat$ there is a canonical equivalence $$\Mor_{\infty\CAT'}(\A,\B) \simeq \{(\A,\B)\} \times_{\infty\CAT'_{[0]} \times \infty\CAT'_{[0]}} \infty\CAT'_{[1]} \simeq \{(\A,\B)\} \times_{\infty\CAT_{[0]} \times \infty\CAT_{[0]}} \infty\CAT_{[1]}.$$
By the dual of Corollary \ref{fib} there is a canonical equivalence $$\{(\A,\B)\} \times_{\infty\CAT_{[0]} \times \infty\CAT_{[0]}} \infty\CAT_{[1]} \simeq \Fun(\A,\B).$$
In particular, for any $\A,\B,\C \in \infty\Cat$ there is a canonical equivalence
$$ \Mor_{\infty\CAT'}(\A \times \B,\C) \simeq \Mor_{\infty\CAT'}(\A, \Mor_{\infty\CAT'}(\B, \C))$$ so that $\A \times \B$ is the tensor of $\A$ and $\B$ in the sense of \cite[Definition 2.51.]{heine2024higher} of the $\infty\Cat$-enriched $\infty$-category $\infty\CAT'$. Thus $\infty\CAT'$ admits all tensors.
%This implies that $\infty\CAT'$ is left tensored over $\infty\Cat$ and so 
By \cite[Remark 3.63., Corollary 6.13.]{HEINE2023108941} this implies that there is a unique 2-functor $\rho: \infty\Cat \to \infty\CAT'$ that preserves tensors and sends the final $\infty$-category to the final $\infty$-category.
By \cite[Corollary 4.50.]{heine2024bienriched} the 2-functor $\rho$ is an equivalence since the 2-functor $\Mor_{\infty\CAT'}(*,-): \infty\Cat \to \infty\Cat$ preserves tensors.

%The 2-functor $\rho$ sends any $\infty$-category $\A \simeq \A \times *$ to $\A \times * \simeq \A$ and so is essentially surjective.Moreover $\rho$ induces on morphism $\infty$-categories a $\infty\Cat$-enriched natural transformation $\theta: \Mor_{\infty\Cat}(-,-) \to \Mor_{\infty\CAT'}(\rho(-),\rho(-))$of 2-functors $\infty\Cat^\op \times \infty\Cat^\op \to \infty\Cat.$For every $\B \in \infty\Cat$ the induced 2-transformation $\theta_{(-,\B)} \Mor_{\infty\Cat}(-,\B) \to \Mor_{\infty\CAT'}(\rho(-),\rho(\B))$of 2-functors $\infty\Cat^\op \to \infty\Cat$ is between 2-functors preserving cotensorsand so an equivalence if $\theta_{(*,\B)}: \Mor_{\infty\Cat}(*,\B) \to \Mor_{\infty\CAT'}(\rho(*),\rho(\B)) $ is an equivalence.The induced 2-transformation $\theta_{(*,-)} \Mor_{\infty\Cat}(*,-) \to \Mor_{\infty\CAT'}(\rho(*),\rho(-))$of 2-functors $\infty\Cat \to \infty\Cat$ is between 2-functors preserving cotensorsand so an equivalence if $\theta_{(*,*)}: \Mor_{\infty\Cat}(*,*) \to \Mor_{\infty\CAT'}(\rho(*),\rho(*)) $ is an equivalence.The latter is an equivalence since source and target identify with $\Fun(*,*)$ and so are contractible.Consequently, $\rho$ is an equivalence.

\end{proof}

\begin{notation}Let $ \infty\widehat{\Cat}$ be the $(\infty,2)$-category of large $\infty$-categories.
Let $\Pr^\L \subset \infty\widehat{\Cat}$ be the subcategory of the $(\infty,2)$-category
$\infty\widehat{\Cat}$ of presentable $\infty$-categories and left adjoints. % functors.
	
\end{notation}

The following remark gives another interpretation of the $(\infty,2)$-category $\Pr^\L:$

\begin{remark}
	
By \cite[Proposition 4.8.1.15.]{lurie.higheralgebra} the underlying $\infty$-category of $\Pr^\L$ admits a closed symmetric monoidal structure such that the inclusion $\bj: \Pr^\L \subset \infty\widehat{\Cat}$ of $\infty$-categories refines to a lax symmetric monoidal functor.
Being symmetric monoidal closed $\Pr^\L$ is enriched in itself and so by transfer of enrichment along $\bj$ also enriched in $\infty\widehat{\Cat}$.
The lax symmetric monoidal inclusion $\bj$ gives rise to an inclusion of $(\infty,2)$-categories from $\Pr^\L$ - endowed with the transfered $\infty\widehat{\Cat}$-enrichment - to 
$\infty\widehat{\Cat}$ endowed with the enrichment in itself by the internal hom,
and so induces an equivalence from the transfered $\infty\widehat{\Cat}$-enrichment on $\Pr^\L$ to
the subcategory of the $(\infty,2)$-category $\infty\widehat{\Cat}$ of presentable $\infty$-categories and left adjoint functors.

\end{remark}

\begin{corollary}\label{amho}

The $(\infty,2)$-category underlying the double $\infty$-category $\PR^\L$ is
$\Pr^\L$. % viewed as enriched in itself via the internal hom.	

\end{corollary}

\begin{corollary}There is a canonical pullback square of $(\infty,2)$-precategories:
$$
\begin{xy}
\xymatrix{
\mathrm{Corr}
\ar[d]
\ar[rr]
&& \mathrm{Pr}^\L \ar[d]^{} 
\\
\Delta_{\iota(\infty\Cat)} \ar[rr]  && \Delta_{\iota(\Pr^\L)}.}
\end{xy}$$
%of cartesian fibrations over $\Delta$ classifying Segal $\infty$-categories.

\end{corollary} 

\begin{corollary}\label{compl} The completion of the $(\infty,2)$-precategory $\Corr$ is the full subcategory of the $(\infty,2)$-category $\Pr^\L$ spanned by the $\infty$-categories of presheaves.
	
\end{corollary}

%\begin{proposition}Let $\mA$ be a small $\infty$-category.Taking pullback along the embedding $\mA \subset \N\Env(\N(\mA))$induces an equivalence $$ \Cat_{/\N\Env(\N(\mA))}^{\mathrm{cart}} \simeq \mathrm{Lo}\Cart_{\mA}.$$$$ \Cat_{/\N\Env(\N(\mA))}^{\mathrm{cart}} \simeq \Fun^\lax(\N\Env(\N(\mA)), \infty\Cat)'$$$$ \Fun^\lax(\N\Env(\N(\mA)), \infty\Cat)' \simeq \FUN(\N\Env(\N(\mA)), \infty\Cat) \simeq \Lax\Fun(\N(\mA), \infty\Cat).$$\end{proposition}

\begin{corollary}\label{genstr} Let $\mC$ be a small $\infty$-category.
There is a canonical equivalence of $\infty$-categories
$$ \FUN(\N(\mC), \CORR) \simeq \mathrm{EXP}_{\mC^\op}.$$	
\end{corollary}

\begin{proof} By Corollary \ref{corflat} and Proposition \ref{oho3} there is a canonical equivalence
$$ \FUN(\N(\mC), \CORR) \simeq \FUN(\N(\mC), \Delta_{\infty\Cat}) \times_{\FUN(\N(\mC), \Delta_{\Pr^\L})}\FUN(\N(\mC), \PR^\L) \simeq $$
$$\Fun(\iota(\mC), \infty\Cat) \times_{\Fun(\iota(\mC),\Pr^\L)}\FUN(\N(\mC), \PR^\L) \simeq $$
$$\infty\Cat_{/\iota(\mC)} \times_{\mathrm{CART}^\L_{\iota(\mC)} } \mathrm{CART}^\L_{\mC^\op} \simeq \mathrm{EXP}_{\mC^\op}.$$

\end{proof}

In the following we prove a refinement of Corollary \ref{genstr}.
\begin{notation}
Let $\phi: \mA \to \Delta$ be a cartesian fibration.
Let $\mA^\dual \to \Delta^\op$ be the corresponding cocartesian fibration classifying the same functor as $\phi.$	
	
\end{notation}

\begin{definition}
Let $\mA \to \Delta, \mB \to \Delta	$ be double $\infty$-categories.

\begin{itemize}
\item A lax map of double $\infty$-categories $\mA \to \mB$ is a functors $\mA^\dual \to \mB^\dual$ over $\Delta^\op$ preserving cocartesian lifts of inert morphisms of $\Delta$
(Definition \ref{inert}).

\item A lax normal map of double $\infty$-categories $\mA \to \mB$ is a functors $\mA^\dual \to \mB^\dual$ over $\Delta^\op$ preserving cocartesian lifts of inert morphisms of $\Delta$ and cocartesian lifts of morphisms of the form $[\n] \to [0]$ for $\n \geq 0$.	

\end{itemize}
\end{definition}

The embedding $\CORR \subset \CORR_\Lax$ of cartesian fibrations over $\Delta$ induces an embedding $\CORR^\dual \subset \CORR^\dual_\Lax$ of cocartesian fibrations over $\Delta^\op.$
Proposition \ref{jjkl} and \cite[Proposition 7.3.2.6.]{lurie.higheralgebra} give the following corollary:
\begin{corollary}
	
The embedding $\CORR^\dual \subset \CORR_\Lax^\dual$ of cocartesian fibrations over $\Delta^\op$
admits a right adjoint relative to $\Delta^\op.$ 
	
\end{corollary}

\begin{example}
The relative right adjoint $\CORR_\Lax^\dual \to \CORR^\dual$ preserves cocartesian lifts of inert morphisms because the left adjojnt does.
So the right adjoint determines a lax map $\CORR_\Lax \to \CORR$ of double $\infty$-categories,
which is a lax normal map.
\end{example}

\begin{remark}
Let $\mA \to \Delta, \mB \to \Delta	$ be double $\infty$-categories.
There is a canonical equivalence $$\FUN(\mA,\mB) \simeq \Fun^\cocart_{\Delta^\op}(\mA^\dual,\mB^\dual)$$ to the full subcategory of $ \Fun_{\Delta^\op}(\mA^\dual,\mB^\dual)$ spanned by the maps $\mA^\dual \to \mB^\dual $ of cocartesian fibrations over $\Delta^\op$.	
	
\end{remark}

\begin{notation}
Let $\mA \to \Delta, \mB \to \Delta	$ be double $\infty$-categories.

\begin{itemize}
\item Let $$\Lax\Fun(\mA,\mB) \subset \Fun_{\Delta^\op}(\mA^\dual,\mB^\dual)$$ the full subcategory spanned by the lax maps of double $\infty$-categories $\mA \to \mB.$ 

\item Let $$\N\Lax\Fun(\mA,\mB) \subset \Fun_{\Delta^\op}(\mA^\dual,\mB^\dual)$$ the full subcategory spanned by the lax normal maps of double $\infty$-categories $\mA \to \mB.$ 
\end{itemize}
\end{notation}

\begin{notation}
Let $\Act \subset \Fun([1],\Delta)$ be the full subcategory of active morphisms, i.e. order preserving maps $[\n]\to [\m]$ preserving the minimum and maximum.	
	
\end{notation}

%\begin{notation}For every double $\infty$-category $\mC \to \Delta$ let $\Env(\mC)$be the pullback $$ \Act^\op \times_{\Delta^\op} \mC^\dual$$ along the opposite of evaluation at the source $\Act \to \Delta$.\end{notation}

Let $\mC \to \Delta $ be a double $\infty$-category.
By \cite[Proposition 3.102.]{HEINE2023108941new} evaluation at the target \begin{equation}\label{eqq}
\Act^\op \times_{\Delta^\op} \mC^\dual \to \Act^\op \to \Delta^\op\end{equation} is a cocartesian fibration, where the pullback is formed along evaluation at the source. % $\Act^\op \to \Delta^\op$.

\begin{notation}Let $\mC \to \Delta $ be a double $\infty$-category.
Let $\Env(\mC) \to \Delta$ be the cartesian fibration corresponding to the cocartesian fibration (\ref{eqq}).
\end{notation}

By \cite[Proposition 3.102.]{HEINE2023108941new} the cartesian fibration $\Env(\mC) \to \Delta$ is a double $\infty$-category
and for every double $\infty$-category $ \mD \to \Delta$ the induced functor $\FUN(\Env(\mC),\mD) \to \Lax\Fun(\mC,\mD)$ is an equivalence.

%and the corresponding cartesian fibration $\Env(\mC)^\dual \to \Delta$ is a double$\infty$-category.
\begin{example}
The diagonal embedding $\Delta \subset \Act$ induces an embedding $\mC^\dual \subset \Act^\op \times_{\Delta^\op} \mC^\dual$ over $\Delta^\op$ that determines a lax map of double $\infty$-categories $\mC \to \Env(\mC).$

\end{example}

\begin{remark}
	
By \cite[Lemma 3.103.]{HEINE2023108941new} the embedding $\mC^\dual \subset \Act^\op \times_{\Delta^\op} \mC^\dual$ over $\Delta^\op$
admits a left adjoint relative to $\Delta^\op$, which is a map of cocartesian fibrations over
$\Delta^\op.$
So this left adjoint corresponds to a map $\Env(\mC) \to \mC$ of cartesian fibrations over
$\Delta$, in other words to a map of double $\infty$-categories.

\end{remark}

\begin{remark}\label{son} Let $\mC \to \Delta$ be a double $\infty$-category.	
There is a canonical equivalence $\Env(\mC)_{[0]} \simeq \mC^\dual_{[0]} \simeq \mC_{[0]}$ since $\Act_{[0]/} $ is contractible.
Consequently, $\Env(\mC) \to \Delta$ is an $(\infty,2)$-precategory if $\mC \to \Delta$ is an $(\infty,2)$-precategory.	
	
\end{remark}

\begin{remark}\label{sap}
Let $\mC \to \Delta$ be a double $\infty$-category.	
Every morphism of $\Env(\mC)$, i.e. object of $\Env(\mC)_{[1]}$,
lies in the image of the restricted composition functor
$$ \mC_{[1]} \times_{\mC_{[0]}}  ... \times_{\mC_{[0]} }\mC_{[1]} \subset \Env(\mC)_{[1]} \times_{\Env(\mC)_{[0]}}  ... \times_{\Env(\mC)_{[0]} }\Env(\mC)_{[1]} \to \Env(\mC)_{[1]}.$$
In other words, any horizontal morphism of $\Env(\mC)\to \Delta$
is of the form $\alpha_1 \circ ... \circ \alpha_\n$ for some $\n \geq 0$ and composable horizontal morphisms $\alpha_1 ,..., \alpha_\n$ of $\mC$, where $\circ$ is the horizontal composition.

\end{remark}

\begin{remark}\label{iww}
Let $\mC \to \Delta$ be an $(\infty,2)$-category.	
Then $\Env(\mC) \to \Delta$ is an $(\infty,2)$-category:
by Remark \ref{son} the canonical map $\iota(\Env(\mC)_{[0]}) \to \iota(\Env(\mC)_{[1]})$
factors as $\iota(\Env(\mC)_{[0]}) \simeq \iota(\mC_{[0]}) \to \iota(\mC_{[1]}) \subset \iota(\Env(\mC)_{[1]}) $
and so is an embedding if $\mC $ is an $(\infty,2)$-category.
Hence in view of Remark \ref{sap} it suffices to observe that for every $\n, \m \geq 0$
and composable horizontal morphisms $\alpha_1,..., \alpha_\n, \beta_1 ,..., \beta_\m$ of $ \mC\to \Delta$ such that the horizontal composition $\alpha_1 \circ ... \circ \alpha_\n \circ \beta_1 \circ ... \circ \beta_\m$ in $\Env(\mC) \to \Delta$ is the 0-fold composition, we find that $\n+\m=0$ so that $\n=\m=0.$

\end{remark}

\begin{notation}
Let $\phi: \mC \to \Delta$ be an $\infty$-precategory and $\mW$ a set of morphisms of $\mC_{[1]}$ %lying over the identity of $[1]$ 
corresponding to 2-morphisms of $\mC \to \Delta$.
For every $\n \geq 1$ let $\mW_\n$ be the set of morphisms of $\mC_{[\n]}$ whose image under the functor
$\mC_{[\n]}\to \mC_{[1]} $ induced by the map $[1] \cong \{\bi-1<\bi\}\to [\n]$ for any $0 \leq \bi < \n$ is a morphism of $\mW.$ Let $\mW_0$ be the set of all morphisms of $\mC_{[0]}.$
Let $\bar{\mW} $ be the smallest set of morphisms of $\mC$ containing all images of morphisms of $\mW_\n$ in $\mC$ for $\n \geq 0.$
%$[\n] \in \Delta$ whose image under the functor$\mC_{[\n]}\to \mC_{[1]} $ induced by the map $[1] \cong \{\bi-1<\bi\}\to [\n]$ for some $0 \leq \bi < \n$ is a morphism of $\mW.$
%\times_{\mC_{[0]} ... \times_{\mC_{[0]}}\mC_{[1]}$ to a compatible family of $\n$-morphisms of $\mW.$
\end{notation}
Let $\phi: \mC \to \Delta$ be an $\infty$-precategory and $\mW$ a set of morphisms of $\mC_{[1]}$.
Let $\mC[\mW]:= \mC[\bar{\mW}]$ be the localization of the $\infty$-category $\mC$ with respect to the set $\bar{\mW}$.
The functor $\phi: \mC \to \Delta$ induces a functor $\mC[\mW] \to \Delta$, which by \cite[2.1.4. Proposition]{HinichDwyer} is a cartesian fibration whose fiber over any $[\n]\in \Delta$ is
the localization $\mC_{[\n]}[\mW_\n].$
In particular, $\mC[\mW]_{[0]} \simeq \mC_{[0]}[\mW_0] \simeq \mC_{[0]}.$
By the next lemma the cartesian fibration $\mC[\mW] \to \Delta$ is an $\infty$-precategory.

%for every double $\infty$-category $\mC \to \Delta$ and every set $\mW$ of morphisms of $\mC$ lying over the identity of $[1]$ corresponding to2-morphisms of $\mC \to \Delta$
%corresponding to morphisms of the $\infty$-category $\mC_{[1]}$the localization $\mC[\mW]$ of the $\infty$-category $\mC$ with respect to $\mW$ % the image of $\mW$ under the projection $there is a cartesian fibration $\mC[\mW] \to \Delta$

\begin{lemma}Let $\mC$ be a space, $\mA \to \mC, \mB \to \mC$ functors and $\mW $ a set of morphisms of $\mA$ and $\mV$ a set of morphisms of $\mB$ containing all equivalences.
The canonical functor $\theta: (\mA \times_\mC \mB)[\mW \times_\mC \mV] \to \mA[\mW]\times_\mC \mB[\mV]$ is an equivalence.	
\end{lemma}

\begin{proof}
	
The functor $\theta$ is a functor over the space $\mC$ and so is an equivalence if it induces an equivalence on the fiber over every $\Z \in \mC.$ 
For every $\Z \in \mC$ let $\mW_\Z$ be the set of morphisms of $\mA_\Z$ whose image in $\mA$ belongs to $\mW.$ Similarly, we define $\mV_\Z.$
The functor $\theta$ induces on the fiber over
$\Z$ the canonical functor $$ (\mA \times_\mC \mB)[\mW \times_\mC \mV]_\Z \to \mA[\mW]_\Z \times \mB[\mV]_\Z,$$
which by \cite[2.1.4. Proposition]{HinichDwyer} identifies with the canonical functor
$$ (\mA_\Z \times \mB_\Z)[\mW_\Z \times \mV_\Z] \to \mA_\Z[\mW_\Z] \times \mB_\Z[\mV_\Z].$$
The latter is an equivalence by universal property of the localization using that
$\mW,\mV$ contain all equivalences.

\end{proof}

\begin{definition}
	
Let $\mC \to \Delta$ be an $\infty$-precategory.	
Let $\N\Env(\mC) \to \Delta$ be the localization of $\Env(\mC) \to \Delta$ with respect to
the smallest set of $\Env(\mC)_{[1]}$ containing the morphisms $\id^{\Env(\mC)}_\X \to \id^\mC_\X $ of $\Env(\mC)_{[1]}$ for $\X \in \mC_{[0]}$ and all equivalences of $\Env(\mC)_{[1]}$
and closed under horizontal composition of the $\infty$-precategory $\Env(\mC) \to \Delta.$
	
\end{definition}

\begin{remark}
	
The localization $\Env(\mC)^\dual \to \mC^\dual$ relative to $\Delta^\op$
whose left adjoint inverts local equivalences, gives rise to a localization $\N\Env(\mC)^\dual \to \mC^\dual$ relative to $\Delta^\op.$
Hence the right adjoint is an embedding $\mC^\dual \subset \N\Env(\mC)^\dual$ that factors as
$\mC^\dual \to \Env(\mC)^\dual \to \N\Env(\mC)^\dual$ and so corresponds to
a lax normal map $\mC \to \N\Env(\mC)$ of double $\infty$-categories.

\end{remark}

\begin{remark}\label{sono}
	
By definition the canonical functor $\Env(\mC) \to \N\Env(\mC)$ over $\Delta$
induces an equivalence $\Env(\mC)_{[0]} \simeq \N\Env(\mC)_{[0]}.$
Thus the embedding $\mC^\dual \subset \N\Env(\mC)^\dual$ over $\Delta$ induces an equivalence
$\mC_{[0]} \simeq \N\Env(\mC)_{[0]}.$
%Thus $\N\Env(\mC) \to \Delta$ is an $(\infty,2)$-precategory if $\mC \to \Delta$ is an $(\infty,2)$-precategory.
	
\end{remark}

\begin{remark}
Let $\mC \to \Delta$ be an $(\infty,2)$-category.	
Then $\N\Env(\mC) \to \Delta$ is an $(\infty,2)$-category by a similar argument like in Remark \ref{iww}: by Remark \ref{sono} the canonical map $\iota(\N\Env(\mC)_{[0]}) \to \iota(\N\Env(\mC)_{[1]})$
factors as $\iota(\N\Env(\mC)_{[0]}) \simeq \iota(\mC_{[0]}) \to \iota(\mC_{[1]}) \subset \iota(\N\Env(\mC)_{[1]}) $
and so is an embedding if $\mC $ is an $(\infty,2)$-category.
So it suffices to observe that for every $\n, \m \geq 0$ and composable horizontal morphisms $\alpha_1,..., \alpha_\n, \beta_1 ,..., \beta_\m$ of $ \mC\to \Delta$ different from identities
such that the horizontal composition $\alpha_1 \circ ... \circ \alpha_\n \circ \beta_1 \circ ... \circ \beta_\m$ in $\N\Env(\mC) \to \Delta$ is the 0-fold composition, we have that $\n+\m=0$ and so $\n=\m=0.$
	
\end{remark}

\begin{remark}\label{uni} Let $\mC \to \Delta, \mD \to \Delta$ be double $\infty$-categories.
By definition the induced functor $\FUN(\N\Env(\mC),\mD) \to \FUN(\Env(\mC),\mD)$ is fully faithful and the equivalence $$\FUN(\Env(\mC),\mD) \simeq \Lax\Fun(\mC,\mD) $$ restricts to an 
equivalence $$\FUN(\N\Env(\mC),\mD) \to \N\Lax\Fun(\mC,\mD).$$

\end{remark}

\begin{proposition}\label{sss} Let $\mC$ be a small $(\infty,2)$-category. 
There is a canonical equivalence $$\Lax\Fun(\mC, \infty\CAT) \simeq \mathrm{Lo}\mathrm{CART}_{\tau^*\mC}.$$
\end{proposition}

\begin{proof}
	
There is a chain of canonical equivalences of $\infty$-categories $$\Lax\Fun(\mC, \infty\CAT) \simeq \FUN(\N\Env(\mC), \infty\CAT) \simeq \mathrm{CART}_{\tau^*\N\Env(\mC)} \simeq$$$$ \mathrm{CART}_{\N\Env(\tau^*\mC)} \simeq \mathrm{Lo}\mathrm{CART}_{\tau^*\mC}.$$
The first equivalence holds by the universal property of $\N\Env(\mC) \to \Delta$ of Remark \ref{uni}. The second equivalence is straightening of cartesian fibrations.
The third equivalence is induced by the canonical equivalence $\tau^*\N\Env(\mC) \simeq \N\Env(\tau^*\mC)$ over $\Delta$, which follows immediately from the construction of $\Env(\mC)\to \Delta.$ The last equivalence is \cite[Theorem B.4.3.]{Ayala2019StratifiedNG}.

\end{proof}

\begin{corollary} Let $\mC$ be a small $\infty$-category. There is a canonical equivalence
$$ \Lax\Fun(\N(\mC), \CORR) \simeq \infty\Cat_{/\mC^\op}.$$	
	
\end{corollary}

\begin{proof}

%There is a canonical equivalence of $\infty$-categories$$ \Lax\Fun(\N(\mC), \Corr) \simeq \infty\Cat_{/\mC^\op}$$	that restricts to an equivalence $$ \FUN(\N(\mC), \Corr) \simeq \mathrm{EXP}_{\mC^\op}.$$	
By Proposition \ref{sss} (applied to a larger universe) there is a canonical equivalence $$\Lax\Fun(\N(\mC), \infty\widehat{\CAT}) \simeq \mathrm{Lo}\widehat{\mathrm{CART}}_{\mC^\op}$$
that restricts to an equivalence
$$\Lax\Fun(\N(\mC), \PR^\L) \simeq \mathrm{LoCART}^\L_{\mC^\op}.$$
Hence by Corollary \ref{corflat} and Proposition \ref{oho3} we obtain a canonical equivalence 
$$ \Lax\Fun(\N(\mC), \CORR) \simeq \Lax\Fun(\N(\mC), \Delta_{\infty\Cat}) \times_{\Lax\Fun(\N(\mC), \Delta_{\Pr^\L})}\Lax\Fun(\N(\mC), \PR^\L) \simeq $$
$$\Fun(\iota(\mC), \infty\Cat) \times_{\Fun(\iota(\mC),\Pr^\L)}\Lax\Fun(\N(\mC), \PR^\L) \simeq $$
$$\infty\Cat_{/\iota(\mC)} \times_{\mathrm{LoCART}^\L_{\iota(\mC)} } \mathrm{LoCART}^\L_{\mC^\op} \simeq \infty\Cat_{/\mC^\op}.$$		

\end{proof}

\begin{example}\label{exoo}
Let $\mA \to \mC$ be an exponential fibration and $\mB \to \mC$ a functor.
By the defining property of exponentiability the functor
$\mA \times_\mC (-) :  \infty\Cat_{/\mC} \to \infty\Cat_{/\mC}$ admits a right adjoint
$\Fun^\mC(\mA,-).$	
The internal hom $\Fun^{\mC}(\mA,\mB)\to \mC $ classifies a lax normal functor $\mC^\op \to  \CORR$
that sends any morphism $\theta: X \to Y$ in $\mC$
to a pro-functor $$  \Fun^{\mC}(\mA,\mB)_Y \simeq \Fun(\mA_Y,\mB_Y) \to \Fun^{\mC}(\mA,\mB)_X \simeq \Fun(\mA_X,\mB_X)$$ corresponding to a functor
$\gamma: \Fun(\mA_X,\mB_X)^\op \times \Fun(\mA_Y,\mB_Y)  \to \mS$.
The functors $\mA \to \mC, \mB \to \mC$ classify lax normal functors
$\mC^\op \to \CORR$, where the first classifies a functor, that send the morphism $\theta: X \to Y$ to pro-functors
$ \A_Y \to \mA_X, \mB_Y \to \mB_X$ corresponding to functors $\alpha: \mA_X^\op \times \mA_Y \to \mS, \beta: \mB_X^\op \times \mB_Y \to \mS.$
Let $\rho: \mW \to  \mA_X^\op \times \mA_Y$ be the left fibration classified by $\alpha$.
By construction of the Grothendieck-construction 
the functor $\gamma$ sends any pair $(F,G) \in \Fun(\mA_X,\mB_X)^\op \times \Fun(\mA_Y,\mB_Y)$
to the mapping space $ \Fun^{[1]}([1] \times_\mC \mA, [1] \times_\mC \mB)(F,G),$
where the pullbacks are taken along $\theta.$
%For every functor $\mM \to [1]$ and $U \in \mM_0, V \in \mM_1$
%there is a canonical equivalence of spaces $\mM(U,V) \simeq \{(U,V) \} \times_{\mM_0 \times \mM_1}
%\Fun_{[1]}([1],\mM).$
By Theorem \ref{mapinner} the mapping space $ \Fun^{[1]}([1] \times_\mC \mA, [1] \times_\mC \mB)(F,G)$
%Consequently, there is a canonical equivalence of spaces
%$$\Fun^{[1]}([1] \times_\mC \mA, [1] \times_\mC \mB)(F,G) \simeq  \{(F,G) \} \times_{ \Fun(\mA_X,\mB_X) \times \Fun(\mA_Y,\mB_Y)}
%\Fun_{[1]}([1] \times_\mC \mA, [1] \times_\mC \mB).$$
is the limit of the functor $\beta \circ (F^\op \times G) \circ \rho: \mW \to \mS. $

\end{example}

%\begin{corollary}
%Let $\mC \to \mA,\mD \to \mA$ be functors, $\f: \s \to \B$ a morphism in $\mA$ and $\F: \mC_\s \to \mD_\s, \G: \mC_\B \to \mD_\B$ functors. There is a canonical equivalence
%$$\{\f\}\times_{\mA(\s,\B)} \Fun^\mA(\mC,\mD)(\F,\G) \simeq ([1]\times_ \mA \Fun^\mA(\mC,\mD))(\F,\G) \simeq$$$$ \Fun^{[1]}([1]\times_ \mA \mC,[1]\times_ \mA \mD)(\F,\G) \simeq $$$$ \lim(\mC_\s^\op \times \mC_\B \times_{([1]\times_ \mA \mC)^\op \times ([1]\times_ \mA \mC)} \Tw([1]\times_ \mA \mC) \to \mC_\s^\op \times \mC_\B \xrightarrow{\F^\op \times \G} \mD_\s^\op \times \mD_\B$$$$ \to ([1]\times_ \mA \mD)^\op \times ([1]\times_ \mA \mD) \xrightarrow{([1]\times_ \mA \mD)(-,-)} \mS).$$\end{corollary}

\bibliographystyle{plain}

\bibliography{ma}

\end{document}